\documentclass[11pt,reqno]{amsart}

\usepackage{fullpage, amsmath, amsthm, amstext, amssymb, amsfonts, fancyhdr, graphicx, tikz-network, forest, mathrsfs, array, mathtools, bm, url, subfigure, caption, float, appendix, hyperref}

\floatstyle{plaintop}
\restylefloat{table}

\NewDocumentCommand\DownArrow{O{2.0ex} O{black}}{%
   \mathrel{\tikz[baseline] \draw [<->, line width=0.5pt, #2] (0,0) -- ++(0,#1);}
}

\allowdisplaybreaks[1]

\newtheorem{theorem}{Theorem}[section]
\newtheorem{lemma}[theorem]{Lemma}
\newtheorem{corollary}[theorem]{Corollary}

\newtheorem{proposition}[theorem]{Proposition}

\theoremstyle{example}
\newtheorem{example}[theorem]{Example}

\newtheoremstyle{boldremark}
    {\dimexpr\topsep/2\relax} 
    {\dimexpr\topsep/2\relax} 
    {}          
    {}          
    {\bfseries} 
    {.}         
    {.5em}      
    {}          

\theoremstyle{boldremark}
\newtheorem{remark}[theorem]{Remark} 

\newtheorem{manualtheoreminner}{Theorem}
\newenvironment{manualtheorem}[1]{%
  \renewcommand\themanualtheoreminner{#1}%
  \manualtheoreminner
}{\endmanualtheoreminner}

\DeclareMathOperator{\PF}{\mathrm{PF}}

\DeclareMathOperator{\PPF}{\mathrm{PPF}}
\DeclareMathOperator{\UPF}{\mathrm{UPF}}
\DeclareMathOperator{\unl}{\mathrm{unl}}
\DeclareMathOperator{\lucky}{\mathrm{lucky}}
\DeclareMathOperator{\dis}{\mathrm{dis}}
\DeclareMathOperator{\des}{\mathrm{idoc}}
\DeclareMathOperator{\idoc}{\mathrm{idoc}}
\DeclareMathOperator{\nld}{\mathrm{nld}}
\DeclareMathOperator{\lev}{\mathrm{lev}}
\DeclareMathOperator{\inv}{\mathrm{inv}}
\DeclareMathOperator{\rlm}{\mathrm{rlm}}
\DeclareMathOperator{\lel}{\mathrm{lel}}
\DeclareMathOperator{\nlel}{\mathrm{nlel}}
\DeclareMathOperator{\one}{\mathrm{ones}}
\DeclareMathOperator{\edes}{\mathrm{edes}}

\newcommand{\sR}{\boldsymbol{s}}

\newcommand{\tauR}{\boldsymbol \tau}
\newcommand{\lambdaR}{\boldsymbol \lambda}
\newcommand{\alphaR}{\alpha}
\newcommand{\betaR}{\beta}

\newcommand{\Sym}{\mathfrak{S}}

\newcommand{\PR}{\mathbb P}

\newcommand{\ER}{\mathbb E}

\newcommand{\calt}{\mathcal{T}}

\newcommand{\Var}{\mathbb V \mathrm{ar}}

\newcommand{\be}{\begin{equation}}
\newcommand{\ee}{\end{equation}}
\newcommand{\sn}{\mathfrak{S}_n}
\newcommand{\oc}{\mathrm{oc}}
\newcommand{\rep}{\mathrm{rep}}
\newcommand{\lead}{\mathrm{ldr}}
\newcommand{\pfn}{\mathrm{PF}(n)}
\newcommand{\beas}{\begin{eqnarray*}}
\newcommand{\eeas}{\end{eqnarray*}}

\newcommand{\old}[1]{}

\title{Some Enumerative Properties of Parking Functions}

\author[Stanley]{Richard P. Stanley}
\address[Stanley]{Department of Mathematics, University of Miami, Coral Gables,
FL 33124}
\email{\textcolor{blue}{\href{mailto:mei.yin@du.edu}{rstan@math.mit.edu}}}

\author[Yin]{Mei Yin}
\thanks{M.~Yin was supported by the University of Denver's Professional Research Opportunities for Faculty Fund 80369-145601.}
\address[Yin]{Department of Mathematics, University of Denver, Denver, CO 80208}
\email{\textcolor{blue}{\href{mailto:mei.yin@du.edu}{mei.yin@du.edu}}}

\begin{document}



\begin{abstract}
A \emph{parking function} is a sequence $(a_1,\dots, a_n)$ 
of positive integers such that if $b_1\leq\cdots\leq b_n$ is the increasing
rearrangement of $a_1,\dots,a_n$, then $b_i\leq i$ for $1\leq i\leq n$.
In this paper we obtain some new results on the enumeration of parking
functions. We will consider the joint distribution of several sets of
statistics on parking functions. The distribution of most of these
individual statistics is known, but the joint distributions are new.
Parking functions of length $n$ are in bijection with labelled forests
on the vertex set $[n]=\{1,2,\dots,n\}$ (or rooted trees on
$[n]_0=\{0,1,\dots,n\}$ with root $0$), so our results can also be
applied to labelled forests. Extensions of our techniques are discussed.
\end{abstract}

\maketitle


\section{Introduction}
In this paper we obtain some new results on the enumeration of parking
functions. We will consider the joint distribution of several sets of
statistics on parking functions. The distribution of most of these
individual statistics is known, but the joint distributions are new.
Parking functions of length $n$ are in bijection with labelled forests
on the vertex set $[n]=\{1,2,\dots,n\}$ (or rooted trees on
$[n]_0=\{0,1,\dots,n\}$ with root $0$), so our results can also be
applied to labelled forests.

We begin with the necessary definitions.  In the classical parking
function scenario due to Konheim and Weiss \cite{k-w}, we have $n$
parking spaces on a one-way street, labelled $1,2,\dots,n$ in
consecutive order as we drive down the street. There are $n$ cars
$C_1,\dots,C_n$. Each car $C_i$ has a preferred space $1\leq a_i\leq
n$. The cars drive down the street one at a time in the order
$C_1,\dots, C_n$. The car $C_i$ drives immediately to space $i$ and
then parks in the first available space. Thus if $i$ is empty, then
$C_i$ parks there; otherwise $C_i$ next goes to space $i+1$, etc. If
all cars are able to park, then the sequence $\pi=(a_1,\dots, a_n)$ is
called a \emph{parking function} of length $n$. It is well-known
and easy to see that if $b_1\leq\cdots\leq b_n$ is the (weakly) increasing
rearrangement of $a_1,\dots,a_n$, then $\pi$ is a parking function if
and only if $b_i\leq i$ for $1\leq i\leq n$. In particular, any
permutation of a parking function is a parking function. Write $\pfn$
for the set of parking functions of length $n$.

The first significant result on parking functions, due to Pyke
\cite{pyke} in another context and then to Konheim and Weiss
\cite{k-w}, is that the number of parking functions of length $n$ is
equal to $(n+1)^{n-1}$. A famous combinatorial proof was given by
Pollak (unpublished but recounted in \cite{Pollak} and
\cite{Pollak2}). It boils down to the following easily verified
statement: let $G$ denote the group of all $n$-tuples
$(a_1,\dots,a_n)\in [n+1]^n$ with componentwise addition modulo
$n+1$. Let $H$ be the subgroup generated by $(1,1,\dots,1)$. Then
every coset of $H$ contains exactly one parking function. Some of our
proofs will be based on generalizations of Pollak's argument.

\forestset{parent color/.style args={#1}{
    {fill=#1},
    for tree={fill/. wrap pgfmath arg={#1!##1}{1/level()*80},draw=#1!80!black}},
    root color/.style args={#1}{fill={{#1!60!black!25},draw=#1!80!black}}
}

Let us review some statistics on parking functions $\pi$ and what was
previously known about their enumeration. For a permutation $b_1 b_2\cdots b_n\in\sn$
(the symmetric group on $[n]$), we define a \emph{descent} to be an index 
$1\leq i\leq n-1$ for which $b_i>b_{i+1}$ 
and a \emph{right-to-left maximum} to be an index $1 \leq i \leq n$
for which $b_i>b_j$ for all $j>i$.

\begin{itemize}

\item Parking outcome: the permutation $\oc(\pi)=b_1 b_2\cdots b_n\in
  \sn$ such that $b_i$ is the space occupied by car $C_i$. Thus if
  $\oc(\pi)^{-1}=c_1 c_2 \cdots c_n$, then $C_{c_i}$ is the car in
  space $i$. For example, if $\pi=(2,2,1,3)$ then $\oc(\pi)=2314$ and
  $\oc(\pi)^{-1}=3124$.
  
\item Displacement: total number of failed attempts before all cars
  find their parking spaces, denoted $\dis(\pi)$. Note that if
  $\pi=(a_1,\dots,a_n)$ and $\oc(\pi) = b_1 \cdots b_n$, then
  $\dis(\pi) = \sum (b_i-a_i) =\binom{n+1}{2}-\sum a_i$. Thus the
  displacement statistic is equivalent to the ``sum of elements''
  statistic $\sum a_i$. It is known \cite[Thm.~1.4]{Yan} that if
    $$ P_n(q) = \sum_{\pi\in \pfn}q^{\dis(\pi)}, $$
  then
   \beas P_1(q) & = & 1,\\
     P_{n+1}(q) & = & \sum_{i=0}^n \binom ni (q^i+q^{i-1}+\cdots+1)
     P_i(q)P_{n-i}(q). \eeas
   From this one can derive the generating function
   \cite[Thm.~1.6]{Yan}: 
    $$ \sum_{n\geq 1}P_n(q)(q-1)^{n-1}\frac{x^n}{n!} =
      \log\sum_{n\geq 0} q^{\binom n2}\frac{x^n}{n!}. $$
 
\item Unluckiness: total number of cars that fail to park at their
  desired spot, denoted $\unl(\pi)$. The complementary statistic
  ``luckiness'' (denoted $\lucky(\pi)$) for parking 
  functions was studied by Gessel and Seo, who gave an explicit
  formula for the corresponding enumerator \cite[Thm.~10.1]{GS}: 
  $$\sum_{\pi \in \PF(n)} q^{\text{lucky}(\pi)}=q\prod_{i=1}^{n-1} (i+(n-i+1)q).$$

\item Descents: total number of descents in the inverse outcome
  $\oc(\pi)^{-1}$, denoted $\des(\pi)$.  For example, $\pi=(2, 2, 1,
  3)$ yields $\oc(\pi)^{-1}=3124$, so
  $\des(\pi)=1$.
  
\item Repeats: total number of cars whose desired spot is the same as 
  that of the previous car, denoted $\rep(\pi)$. Using a bijection
  between parking functions and Pr\"{u}fer code and recognizing that a
  zero in the Pr\"{u}fer code indicates a pair of consecutive like
  numbers in the parking function, an explicit formula for the
  enumerator of parking functions by repeats was established
  \cite[Cor.~1.3]{Yan}: 
  $$\sum_{\pi \in \PF(n)} q^{\rep(\pi)}=(q+n)^{n-1}.$$

\item Leading elements: total number of cars whose desired spot is the
  same as that of the first car, denoted $\lel(\pi)$. This statistic
  has not been considered before. 

\item 1's: total number of cars whose desired spot is spot $1$,
  denoted $\one(\pi)$.

\item Right-to-left maxima: total number of right-to-left maxima in the inverse outcome
  $\oc(\pi)^{-1}$, denoted $\rlm(\pi)$.  For example, $\pi=(2, 2, 1,
  3)$ yields $\oc(\pi)^{-1}=3124$, so
  $\rlm(\pi)=1$.
  
\item Inversions: a pair of vertices $(i,j)$ of a labelled rooted tree
  $T$ with root 0 constitutes an inversion if $i<j$ and $j$ lies on
  the unique path connecting the root vertex to $i$. Let $\calt(n)$
  denote the set of all rooted trees on the vertex set $[n]_0$ with
  root 0. We write $\inv(T)$ for the number of inversions of
  $T\in\calt(n)$. The distribution of inv on trees $T\in\calt(n)$
  coincides with the distribution of dis on parking functions $\pi$ of
  length $n$ \cite[{\S}1.2.2]{Yan}.

\item Leaders (proper vertices): a leader is a vertex of a tree 
  $T\in\calt(n)$ which is the smallest among all the vertices of the
  subtree rooted at this vertex. The root vertex $0$ is not counted,
  either as a leader or as a non-leader. We write $\lead(T)$ for the
  number of leaders of $T$ and $\nld(T)$ for the complementary
  number of non-leaders of $T$. This statistic for trees was studied by
  Gessel and Seo, who gave an explicit formula for the corresponding
  enumerator \cite[Cor.~8]{Seo} \cite[Thm.~6.1]{GS}: 
  $$\sum_{T \in \calt(n)} q^{\lead(T)}=q\prod_{i=1}^{n-1}
  (i+(n-i+1)q).$$ See also Seo and Shin \cite{SS} for a simple
  bijective proof of Gessel and Seo's formula using a variation of
  Pr\"{u}fer code.

\item Leaves: a leaf is a vertex of a tree $T\in\calt(n)$ with no
  children. A single root is also considered to be a leaf.
  We write $\lev(T)$ for the number of leaves of $T$.

\item Root degree: The number of children of the root vertex in a labelled 
rooted tree $T$ with root 0, denoted $\deg_T(0)$.

\item Edge descents: an edge descent of a tree $T\in\calt(n)$ is an
  edge of $T$ whose vertex nearer the root is larger than the other
  vertex. The number of edge descents of $T$ is denoted $\edes(T)$. 
  

\end{itemize}

  Here is a summary of our notation just discussed.

\begin{tabbing}
 aaa\=aaaaaaaa\=aaaaaaaaaaaaaaaaaaaaaaaaaaaaaaaaaaaaaaaa\=aaaaaaaaa\=aaaaaaaaaaaaaaaaaaaaaaaa\=aaaaa\=\kill
 \>$\oc(\pi)$:\> parking outcome \> $\dis(\pi)$:\> displacement\\
 \>$\unl(\pi)$:\> unluckiness \> $\lucky(\pi)$:\> luckiness\\
  \> $\idoc(\pi)$:\> number of descents of
   $\oc(\pi)^{-1}$
 \>$\rep(\pi)$:\> repeats\\ \>$\lel(\pi)$:\> leading elements
 \>$\one(\pi)$:\> 1's \\ \>$\rlm(\pi)$: \> number of right-to-left maxima of
   $\oc(\pi)^{-1}$ \> $\inv(T)$: \> inversions\\ 
 \>$\nld(T)$:\> non-leaders
 \> $\lev(T)$:\> leaves\\  \>$\deg_T(0)$: \> root degree 
 \>$\edes(T)$:\> edge descents\\
\end{tabbing}

\section{A recurrence relation for refined trees and parking functions}

Define
 $$ P_n(x,y,z,w):=\sum_{\pi \in \mathrm{PF}(n)} x^{\unl(\pi)} y^{\dis(\pi)}
   z^{\des(\pi)} w^{\mathrm{rlm}(\pi)}$$
and 
  $$ Q_n(x,y,z,w):=\sum_{T \in \mathcal{T}(n+1)}
    x^{\nld(T)} y^{\inv(T)}
    z^{\lev(T)-1} w^{\mathrm{deg}_T(0)}. $$
We will show that $P_n(x,y,z,w)=Q_n(x,y,z,w)$ by showing that they
satisfy the same recurrence relation. By convention $P_0(x, y, z,
w)=Q_0(x,y,z,w)=1$.  

\begin{theorem}\label{un1-general}
(a) $P_n(x, y, z, w)$ satisfies the following recurrence relation:
\begin{multline*}
P_{n}(x, y, z, w)=(1+xy+\cdots+xy^{n-1})w P_{n-1}(x, y, z, 1)\\ 
+\sum_{i=0}^{n-2} \binom{n-1}{i} (1+xy+\cdots+xy^i)zw P_i(x, y, z, 1)
P_{n-i-1}(x, y, z, w).\ \ \
\end{multline*}
(b) $Q_n(x, y, z, w)$ satisfies the following recurrence relation:
\begin{multline*}
Q_n(x, y, z, w)=(1+xy+\cdots+xy^{n-1})w Q_{n-1}(x, y, z, 1)\\ 
+\sum_{i=0}^{n-2} \binom{n-1}{i} (1+xy+\cdots+xy^i)zw Q_i(x, y, z, 1)
Q_{n-i-1}(x, y, z, w).\ \ \
\end{multline*}
\end{theorem}

\begin{proof}
(a) Cars $1, \dots, n-1$ have all parked along the
street before the $n$th car enters, leaving only one open spot for the
$n$th car to park. We denote this spot by $i+1$, where $i=0, \dots,
n-1$. Since a car cannot jump over an empty spot, the parking protocol
implies that $(\pi_1, \dots, \pi_{n-1})$ may be decomposed into two
parking functions $\alphaR \in \PF(i)$ and $\betaR \in \PF(n-1-i)$,
where $\betaR$ is formed by subtracting $i+1$ from the relevant
entries in $\pi$, and $\alphaR$ and $\betaR$ do not interact with
each other. The binomial coefficient $\binom{n-1}{i}$ chooses an
$i$-element subset of $[n-1]$ to constitute the index set of
$\alphaR$. 
 
See Figure \ref{pf:illustration}. This open spot $i+1$ could be either
the same as $j$, the preference of the last car, in which case the car
parks directly. That is where the $1$ comes from. Or, $i+1$ could be
larger than $j$, in which case the car travels forward to park,
contributing to the increase of both displacement and
unluckiness. That is why we have $xy+xy^2+\cdots+xy^i$. In more
detail, we have: 
\begin{equation*}
\dis(\pi)=\dis(\alphaR)+\dis(\betaR)+(i+1-j),
\end{equation*}
\begin{equation*}
\unl(\pi)=\unl(\alphaR)+\unl(\betaR)+(\mathrm{either}\ 1\ \mathrm{or}\ 0),
\end{equation*}
where it is $1$ if $j<i+1$ and $0$ if $j=i+1$. The $z$ factor comes
into play when the last car does not end up parking at the last spot;
whether its preference is equal to the parking outcome does not
matter. Lastly, for the $w$ factor, we note that $P_i$ corresponds to 
the total number of rlm cars that are parked to the left of the last car, 
and $P_{n-i-1}$ corresponds to the total number of rlm cars that are 
parked to the right of the last car. Now the last car must be an rlm, 
so after it enters, the total number of rlm cars that are parked to the 
left of it does not matter anymore, and the rlm cars of $\pi$ now 
consists of the last car and the rlm cars that are parked to the right of it.

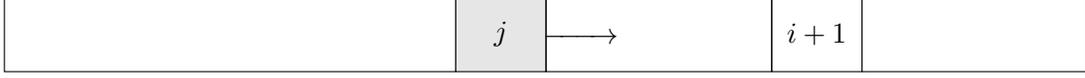
\begin{figure}
\centering
\begin{tikzpicture}
    \def\bh{1} 
    \def\sepw{1pt} 
    \def\ta{6}
    \draw (\ta,0.5 - 0.5 * \bh ) rectangle (\ta+\ta, 0.5 + 0.5 * \bh ) node[pos=.5] {};
    \draw [fill=gray!20](\ta+\ta,0.5 - 0.5 * \bh ) rectangle (\ta+\ta+0.2*\ta, 0.5 + 0.5 * \bh ) node[pos=.5] {$j$};
    \draw (\ta+\ta+0.2*\ta,0.5 - 0.5 * \bh ) rectangle (\ta+\ta+0.2*\ta+0.5*\ta, 0.5 + 0.5 * \bh ) node[pos=.5] {};
    \draw (\ta+\ta+0.2*\ta+0.5*\ta,0.5 - 0.5 * \bh ) rectangle (\ta+\ta+0.2*\ta+0.5*\ta+0.2*\ta, 0.5 + 0.5 * \bh ) node[pos=.5] {$i+1$};
    \draw (\ta+\ta+0.2*\ta+0.5*\ta+0.2*\ta,0.5 - 0.5 * \bh ) rectangle (\ta+\ta+0.2*\ta+0.5*\ta+0.2*\ta+0.5*\ta, 0.5 + 0.5 * \bh ) node[pos=.5] {};
    \node[anchor=north east] at (\ta+\ta+0.383*\ta,0.65) {\parbox{1cm}{\rightarrowfill}};
\end{tikzpicture}
\caption{Parking function case.}
\label{pf:illustration}
\end{figure}

\bigskip

(b) Consider a tree $T$ with $n+1$ vertices $0, 1, \dots, n$
rooted at $0$. If in the unique chain joining $0$ to $n$ we remove the
edge from $0$, which we denote by $0r$, one obtains (a) a tree $T'$
with $n-i$ vertices of which one is $0$, and (b) a tree $T''$ with
$i+1$ vertices of which one is $r$, where $0\leq i\leq n-1$. See
Figure \ref{tree:illustration}. Recall that $n$ is automatically a
vertex of $T''$ by construction. The binomial coefficient
$\binom{n-1}{i}$ chooses $i$ vertices out of $[n+1]_0 \setminus \{0,
n\}$ to form $T''$ and the remaining vertices of $T$ will constitute
$T'$. Changing the vertex labels of the descendants of $0$ in $T'$ and
of $r$ in $T''$ to their relative order among all the descendants of
the tree will not affect the quadruple statistics (non-leaders,
inversions, leaves-1, degree of root vertex) contribution from these vertices. $T'$ is rooted
at $0$ and so the quadruple statistics  contribution from $T'$ is
$Q_{n-i-1}(x, y, z, w)$. $T''$ is rooted at $r$ not $0$, but we could
further change the label of vertex $r$ to $0$ (hence the contribution
of $Q_i(x, y, z, w)$) and then add back the contribution from vertex
$r$. 

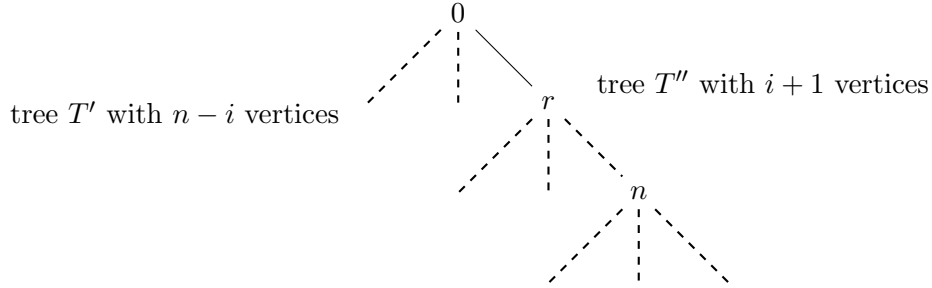
\begin{figure}
\centering
\begin{tikzpicture}[scale=0.8]
\node at (0,0) {$0$} [grow = down]
    child {edge from parent [dashed, thick]}
    child {edge from parent [dashed, thick]}
    child {node {$r$} child{edge from parent
        [dashed, thick]} child{edge from parent [dashed,
        thick]} child{node {
        $n$} edge from parent [dashed, thick] child{edge from parent [dashed,
          thick]} child{edge from parent [dashed,
            thick]} child{edge from parent [dashed, thick]}}};
\node[anchor=north east] at (8, -0.8) {tree $T''$ with $i+1$ vertices};
\node[anchor=north east] at (-1.8,-1.3) {tree $T'$ with $n-i$ vertices};
\end{tikzpicture}
\caption{Tree case.}
\label{tree:illustration}
\end{figure}

The additional factor $1+xy+\cdots+xy^i$ indicates the double
statistics (non-leaders, inversions) contribution from the relative
order of $r$ in tree $T''$. When $r$ is the smallest vertex in $T''$,
it is a leader and does not contribute to inversions. In all other
cases, $r$ is a non-leader and depending on how many descendants of
$r$ are less than $r$, the inversion contribution from $r$ would be
different, ranging from $y$ to $y^i$ as there are $i$ descendants of
$r$ in total. Next we consider the effect of merging the two trees
$T'$ and $T''$ on the single statistic (leaves-1). If $i=n-1$, we are
merely affixing $0$ to $r$ as the new root of the tree and the
(leaves-1) statistic stays the same. Otherwise there is an extra
factor of $z$ because
$\lev(T)-1=\lev(T')+\lev(T'')-1=1+(\lev(T')-1)+(\lev(T'')-1)$.
Lastly, for the $w$ factor, we note that $Q_{n-i-1}$ corresponds 
to the degree of vertex $0$ in tree $T'$ and $Q_i$ corresponds to 
the degree of vertex $r$ in tree $T''$. For the entire tree $T$, 
the degree of vertex $0$ is always the degree of vertex $0$ in 
tree $T'$ plus $1$, the degree of vertex $r$ in tree $T''$ no longer matters.

\smallskip

Note that the above ``break apart'' argument in both the parking function case and
the tree case could be repeated, which corresponds to the
iterative process described by this recurrence formula.
\end{proof}

Under the classical parking protocol, when a car's desired spot is
taken, it always drives forward to look for an available parking
spot. Let us incorporate a probabilistic scenario in the parking
protocol. Fix $p\in [0, 1]$ and consider a coin which flips to heads
with probability $p$ and tails with probability $1-p$. Our
probabilistic parking protocol proceeds as follows: if a car arrives
at its preferred spot and finds it unoccupied it parks there. If
instead the spot is occupied, then the driver tosses the biased
coin. If the coin lands on heads, with probability $p$, the driver
continues moving forward in the street. However, if the coin lands on
tails, with probability $1-p$, the car moves backward and tries to
find an unoccupied parking spot. When $p=1$, the probabilistic parking
model reduces to the classical parking model. See \cite{DHHRY} for
more details of this setup. We will see later that there is a deep
implication of this probabilistic scenario on the parking protocol; it is
not merely an artificial twist that is being added. 

In the case of classical parking functions, a preference vector $\pi
\in [n]^n$ is either deterministically a parking function or not, and
we can talk about the set of parking functions $\PF(n)$ and its
cardinality $\vert\PF(n)\vert$. Contrarily, by incorporating a
probabilistic parking protocol, all preference vectors $\pi \in
[n]^n$ have a positive probability of being a parking function, and we
denote by $\PR(\pi \in \PF(n))$ the probability that $n$ cars with
preference vector $\pi$ park. 

For example, consider parking preference $\pi=(2, 2, 2) \in
[3]^3$. In the classical parking model, $\pi$ is deterministically
not a parking function. However, under the probabilistic scenario, this
determistic conclusion fails. Car $1$ directly parks at its desired
spot $2$. Now car $2$ enters, its preference spot is taken by car $1$,
so car $2$ would either move forward with probability $p$ and park at
spot $3$ or move backward with probability $1-p$ and park at spot
$1$. Next car $3$ enters, in the first situation, it will move
backward with probability $1-p$ to park at spot $1$, while in the
second situation, it will move forward with probability $p$ to park at
spot $3$. We see that both the first situation and the second
situation happen with probability $p(1-p)$, but for different
reasons. Altogether, $\PR((2, 2, 2) \in \PF(3))=2p(1-p)$. 

%

Let
 $$\sum_{\pi} x^{\unl(\pi)} y^{\dis(\pi)} z^{\des(\pi)}
  w^{\rlm(\pi)} \hspace{.1cm} \PR(\pi \in \PF(n)):=P'_n(x, y, z,
   w).$$ 
\begin{manualtheorem}{\ref{un1-general}'}\label{un1-general-alternate}
$P'_n(x, y, z, w)$ satisfies the following recurrence relation:
\begin{multline}\label{general-p}
P'_{n}(x, y, z, w)=(1+p(xy+\cdots+xy^{n-1}))w P'_{n-1}(x, y, z, 1)\\
+\sum_{i=0}^{n-2} \binom{n-1}{i} (1+p(xy+\cdots+xy^i)+(1-p)(xy+\cdots+xy^{n-i-1}))zw P'_i(x, y, z, 1) P'_{n-i-1}(x, y, z, w).
\end{multline}
\end{manualtheorem}

\begin{proof}
The proof is almost identical to the proof of Theorem
\ref{un1-general}. But note that under the probabilistic scenario, when a
car's desired spot is taken, it could either move forward with
probability $p$ or backward with probability $1-p$, there are thus
three possible scenarios for the preference of the last car, which we
denote by $j$, as compared with the open spot, which we denote by
$i+1$. Either $j=i+1$ (giving the contribution of $1$), or $j<i+1$
(giving the contribution of $p(xy+\cdots+xy^i)$), or $j>i+1$ (giving
the contribution of $(1-p)(xy+\cdots+xy^{n-i-1})$). See Figure
\ref{pf-prob:illustration}. 
\end{proof}

\begin{figure}
\centering
\begin{tikzpicture}
    \def\bh{1} 
    \def\sepw{1pt} 
    \def\ta{4.8}
    \draw (\ta,0.5 - 0.5 * \bh ) rectangle (\ta+\ta, 0.5 + 0.5 * \bh ) node[pos=.5] {};
    \draw [fill=gray!20](\ta+\ta,0.5 - 0.5 * \bh ) rectangle (\ta+\ta+0.2*\ta, 0.5 + 0.5 * \bh ) node[pos=.5] {$j$};
    \draw (\ta+\ta+0.2*\ta,0.5 - 0.5 * \bh ) rectangle (\ta+\ta+0.2*\ta+0.5*\ta, 0.5 + 0.5 * \bh ) node[pos=.5] {};
    \draw (\ta+\ta+0.2*\ta+0.5*\ta,0.5 - 0.5 * \bh ) rectangle (\ta+\ta+0.2*\ta+0.5*\ta+0.2*\ta, 0.5 + 0.5 * \bh ) node[pos=.5] {$i+1$};
    \draw (\ta+\ta+0.2*\ta+0.5*\ta+0.2*\ta,0.5 - 0.5 * \bh ) rectangle (\ta+\ta+0.2*\ta+0.5*\ta+0.2*\ta+0.5*\ta, 0.5 + 0.5 * \bh ) node[pos=.5] {};
    \node[anchor=north east] at (\ta+\ta+0.428*\ta,0.65) {\parbox{1cm}{\rightarrowfill}};
    \node[anchor=north east] at (\ta+\ta+0.345*\ta,0.95) {$p$};
    \draw [fill=gray!20] (\ta+\ta+0.2*\ta+0.5*\ta+0.2*\ta+0.5*\ta,0.5 - 0.5 * \bh ) rectangle (\ta+\ta+0.2*\ta+0.5*\ta+0.2*\ta+0.5*\ta+0.2*\ta, 0.5 + 0.5 * \bh ) node[pos=.5] {$j$};
    \draw (\ta+\ta+0.2*\ta+0.5*\ta+0.2*\ta+0.5*\ta+0.2*\ta,0.5 - 0.5 * \bh ) rectangle (\ta+\ta+0.2*\ta+0.5*\ta+0.2*\ta+0.5*\ta+0.2*\ta+0.5*\ta, 0.5 + 0.5 * \bh ) node[pos=.5] {};
    \node[anchor=north east] at (\ta+\ta+0.2*\ta+0.5*\ta+0.2*\ta+0.5*\ta+0.035*\ta,0.65) {\parbox{1cm}{\leftarrowfill}};
    \node[anchor=north east] at (\ta+\ta+0.2*\ta+0.5*\ta+0.2*\ta+0.4*\ta,0.45) {$1$}; \node[anchor=north east] at (\ta+\ta+0.2*\ta+0.5*\ta+0.2*\ta+0.5*\ta,0.45) {$-p$};
\end{tikzpicture}
\caption{Parking function case under probabilistic scenario.}
\label{pf-prob:illustration}
\end{figure}
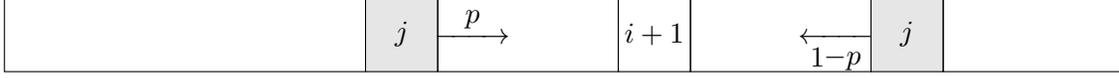

Setting $z=1$ and $w=1$ in (\ref{general-p}) and for simplicity write $P'_n(x, y, 1, 1)$ by $P'_n(x, y)$, interestingly, the $p$ dependence disappears.

\begin{corollary}\label{cancelation}
$P'_n(x, y)$ is independent of $p$ and satisfies the following recurrence relation:
\begin{equation*}
P'_{n}(x, y)=\sum_{i=0}^{n-1} \binom{n-1}{i} (1+xy+\cdots+xy^i) P'_i(x, y) P'_{n-i-1}(x, y).
\end{equation*}
\end{corollary}

\begin{proof}
By convention $P'_0(x, y)=1$. From Theorem \ref{un1-general-alternate}, for $n\geq 1$,
\begin{equation*}
P'_{n}(x, y)=\sum_{i=0}^{n-1} \binom{n-1}{i} (1+p(xy+\cdots+xy^i)+(1-p)(xy+\cdots+xy^{n-i-1})) P'_i(x, y) P'_{n-i-1}(x, y).
\end{equation*}
We note that the index set $\{0, \dots, n-1\}$ for $i$ may be
partitioned into
  $$\{0, n-1\}, \{1, n-2\}, \dots,
    \left\{\frac{n-3}{2}, \frac{n+1}{2}\right\}, \left\{\frac{n-1}{2}\right\} \
   \text{if $n$ is odd},$$
and
   $$\{0, n-1\}, \{1, n-2\}, \dots,
     \left\{\frac{n-2}{2}, \frac{n}{2}\right\} \  \text{if $n$ is even}.$$ 
We group the summands based on the above partition:
\begin{align*}
&\binom{n-1}{i} (1+p(xy+\cdots+xy^i)+(1-p)(xy+\cdots+xy^{n-i-1})) P'_i(x, y) P'_{n-i-1}(x, y)+ \notag \\
&\ \binom{n-1}{n-i-1} (1+p(xy+\cdots+xy^{n-i-1})+(1-p)(xy+\cdots+xy^i)) P'_{n-i-1}(x, y) P'_i(x, y) \notag \\
=&\binom{n-1}{i} \left((1+xy+\cdots+xy^i)+(1+xy+\cdots+xy^{n-i-1})\right) P'_i(x, y) P'_{n-i-1}(x, y).
\end{align*}
When $n$ is odd, there is one extra term:
\begin{align*}
&\binom{n-1}{\frac{n-1}{2}} (1+p(xy+\cdots+xy^{\frac{n-1}{2}})+(1-p)(xy+\cdots+xy^{\frac{n-1}{2}})) P'_{\frac{n-1}{2}}(x, y) P'_{\frac{n-1}{2}}(x, y) \notag \\
=&\binom{n-1}{\frac{n-1}{2}} (1+xy+\cdots+xy^{\frac{n-1}{2}}) P'_{\frac{n-1}{2}}(x, y) P'_{\frac{n-1}{2}}(x, y).
\end{align*}
By induction our proof is complete.
\end{proof}

\begin{remark} 
We verified the cancellation of $p$ dependence in Corollary
\ref{cancelation} via direct computation, but there is a deeper reason
behind this disappearance of $p$ which we will later elaborate upon. 
\end{remark}

Take $m\leq n$. Recall the bijection between parking functions
$\text{PF}(m,n)$ with $m$ cars and $n$ spots and spanning forests
$\mathcal{F}(n+1,n-m+1)$ with $n+1$ vertices and $n-m+1$ distinct
trees having specified roots \cite{KY}: for a parking function $\pi
\in \PF(m, n)$, there are $n-m$ parking spots that are never attempted
by any car. Let $k_i(\pi)$ for $i=1, \dots, n-m$ represent these
spots, so that $0:=k_0<k_1<\cdots < k_{n-m}<k_{n-m+1}:=n+1$. This
separates $\pi$ into $n-m+1$ disjoint noninteracting segments, with
each segment a classical parking function of length $k_i-k_{i-1}-1$
after translation. Each of these classical parking functions
corresponds to a rooted tree in the forest $F \in
\mathcal{F}(n+1,n-m+1)$. From Theorem
\ref{un1-general}, we may let 
\begin{align*}
\sum_{\pi \in \text{PF}(m, n)} x^{\unl(\pi)} y^{\dis(\pi)} z^{\des(\pi)} w^{\rlm(\pi)}
=&\sum_{F \in \mathcal{F}(n+1, n+1-m)} x^{\nld(F)} y^{\inv(F)}
z^{\lev(F)-1} w^{\deg_T(0)}\\:=&P_{m, n}(x, y, z, w). 
\end{align*}
Similarly, under the probabilistic scenario, let
\begin{equation*}
\sum_{\pi} x^{\unl(\pi)} y^{\dis(\pi)} z^{\des(\pi)}
w^{\rlm(\pi)} \hspace{.1cm} \PR(\pi \in \PF(m, n)):=P'_{m, n}(x, y, z,
w). 
\end{equation*}
The following propositions are easy generalizations of Theorem
\ref{un1-general} and Theorem \ref{un1-general-alternate} based upon
the above correspondence, where we take $\sR=(k_1-k_0-1, \dots,
k_{n-m+1}-k_{n-m}-1)$. 

\begin{proposition}
\begin{equation*}
P_{m, n}(x, y, z, w)=\sum_{\sR \in C(m)} \binom{m}{\sR} \prod_{i=1}^{n-m+1} P_{s_i}(x, y, z, w),
\end{equation*}
where $C(m)$ consists of compositions of $m$: $\sR=(s_1, \dots,
s_{n-m+1}) \models m$ with $\sum_{i=1}^{n-m+1} s_i=m$. 
\end{proposition}

\begin{proposition}
\begin{equation*}
P'_{m, n}(x, y, z, w)=\sum_{\sR \in C(m)} \binom{m}{\sR} \prod_{i=1}^{n-m+1} P'_{s_i}(x, y, z, w),
\end{equation*}
where $C(m)$ consists of compositions of $m$: $\sR=(s_1, \dots,
s_{n-m+1}) \models m$ with $\sum_{i=1}^{n-m+1} s_i=m$. 
\end{proposition}


\section{Circular Symmetry: Part I}
\label{sec:sym1}
In this section we derive various results for parking functions $\pi
\in \PF(n)$, using a combination of probabilistic and combinatorial
tools. We will follow Pollak's ingenious circle argument \cite{Pollak}
for the street parking model and assign $n$ cars on a circle with
$n+1$ spots. Those car assignments where spot $n+1$ is left empty
after circular rotation give valid parking functions. A forward (or
backward) movement of a car on a street could be interpreted as a
clockwise (or counterclockwise) movement of a car on a circle. We will
write our results for the parking situation under the probabilistic
scenario, where a car moves forward with probability $p$ and backward
with probability $1-p$ when its desired spot is taken. But note that
when $p=1$ we recover the classical deterministic parking model and so
readers who are less familiar with probability may interpret all
results deterministically. 

Unlike Pollak's original argument where the parking statistics are
studied after all cars have parked, we will investigate the individual
parking statistics for each car the moment it is parked on the
circle. It turns out that this ``seemingly small step ahead'' will
provide a lot more useful information about a range of parking
statistics, such as ``unluckiness'' and ``displacement,'' that we
studied earlier. We note that $\unl(\pi)$ and
$\dis(\pi)$ are sums of individual parking statistics
$\unl(\pi_i)$ and $\dis(\pi_i)$, and more importantly
only the parking preferences of the first $i$ cars matter when
studying the corresponding statistics of the $i$th car that enters. We
also note that the average parking statistics are invariant under
circular rotation and the direction that the car chooses to move when
its preferred spot is occupied. This is the deeper reason behind the
disappearance of the probabilistic parameter $p$ in the generating
function $P'_n(x, y)$ in Corollary \ref{cancelation}.

Based on the above crucial observations, we first present some
probabilistic results for the parking statistics unluckiness and
displacement. 

\begin{proposition}\label{prob_displacement}
Consider a street with $n$ spots, where $m\leq n$ cars are parked under
the probabilistic parking protocol. Take $0\leq i\leq m-1$. The
probability that the displacement of the $(i+1)$th car is $k$, where
$0 \leq k \leq i$, is given by 
\begin{align*}
&P_k:=\PR\left(\text{displacement of } (i+1)\text{th car is } k \hspace{.1cm} \vert \hspace{.1cm} \pi \in \PF(m, n)\right) \notag \\
&=\frac{n-i}{(n+1)^i} \left[p\sum_{j=k+1}^{i+1} \binom{i}{j-1} j^{j-2} (n-j+1)^{i-j}\right. \notag \\
& \hspace{3cm} \left.+(1-p)\sum_{j=n-i}^{n-k} \binom{i}{n-j} (n-j+1)^{n-j-1} j^{i-n+j-1}\right] \notag \\
&=\frac{n-i}{(n+1)^i} \sum_{s=k}^i \binom{i}{s} (s+1)^{s-1} (n-s)^{i-s-1}.
\end{align*} 
The average displacement of parking the $(i+1)$th car is
$D_{i}=\sum_{k=0}^i k P_k$, and the average displacement of parking a
random car is $\frac{1}{m}\sum_{i=0}^{m-1} D_i$. 
\end{proposition}

\begin{proof}
As observed earlier, no matter how many cars are to be parked in
total, only the parking preferences of the first $i+1$ cars have an
effect on the displacement of the $(i+1)$th car that enters. The
minimum possible displacement of the $(i+1)$th car is $0$ and the
maximum possible displacement of the $(i+1)$th car is $i$, which
happens when the first $i$ entering cars occupy an entire block and
the desired spot of the $(i+1)$th car is at one end of the block. 

We assign $i$ cars on a circle with $n+1$ spots. This leaves open
$n-i+1$ spots. One of these would become the street exit in the street
parking model, and any of the other $n-i$ available spots may be
parked in by the $(i+1)$th car. Adapting circular symmetry to the
street parking situation and denoting the available spot that is later
parked in by the $(i+1)$th car by $j$, we only need to consider two
extremal cases: (a) spot $j$ is the leftmost open spot after the first
$i$ cars have parked and the $(i+1)$th car travels forward to park,
and (b) spot $j$ is the rightmost open spot after the first $i$ cars
have parked and the $(i+1)$th car travels backward to park. For case
(a), there are $j-1$ cars and $j-1$ spots to the left of spot $j$ and
$i-j+1$ cars and $n-j$ spots to the right of spot $j$. The fact that
the forward-moving displacement of the $(i+1)$th car is $k$ implies
that $j\geq k+1$. We must also impose that $j\leq i+1$ to ensure that
$i-j+1 \geq 0$. Similarly, for case (b), there are $n-j$ cars and
$n-j$ spots to the right of spot $j$ and $i-n+j$ cars and $j-1$ spots
to the left of spot $j$. The backward-moving displacement of the
$(i+1)$th car is $k$ implies that $j \leq n-k$. We must also impose
that $j\geq n-i$ to ensure that $i-n+j \geq 0$.

From a slight generalization of Corollary \ref{cancelation}, the
non-normalized total probability of valid parking functions under a
probabilistic street parking model with $m$ cars and $n$ spots is
independent of the forward probability $p$ and assumes the same
formula as in the classical setting: 
\begin{equation*}
\sum_{\pi} \PR(\pi \in \PF(m, n))=(n-m+1) (n+1)^{m-1}.
\end{equation*}
Therefore the probability of the first $i+1$ cars successfully parking
is $(n-i)(n+1)^i$. Now case (a) happens with probability
$$p\sum_{j=k+1}^{i+1} \binom{i}{j-1} j^{j-2} (n-i)(n-j+1)^{i-j},$$
where the binomial coefficient $\binom{i}{j-1}$ chooses $j-1$ cars out
of $i$ cars to assign to the left of spot $j$, and the remaining
$i-j+1$ cars are assigned to the right of spot $j$. Similarly, case
(b) happens with probability 
  $$ (1-p)\sum_{j=n-i}^{n-k} \binom{i}{n-j} (n-j+1)^{n-j-1} (n-i)
      j^{i-n+j-1}, $$ 
where the binomial coefficient $\binom{i}{n-j}$ chooses $n-j$ cars out
of $i$ cars to assign to the right of spot $j$ and the remaining
$i-n+j$ cars are assigned to the left of spot $j$.
Recalling that there are $n-i$ available spots after the first $i$
cars have parked, we have 
\begin{align*}
  &\PR\left(\text{displacement of } (i+1)\text{th car is } k \hspace{.1cm}
  \vert \hspace{.1cm} \pi \in \PF(m, n)\right) \notag \\ 
&=\frac{(n-i)^2}{(n-i)(n+1)^i} \left[p\sum_{j=k+1}^{i+1} \binom{i}{j-1} j^{j-2} (n-j+1)^{i-j}\right. \notag \\
& \hspace{3cm} \left.+(1-p)\sum_{j=n-i}^{n-k} \binom{i}{n-j} (n-j+1)^{n-j-1} j^{i-n+j-1}\right] \notag \\
&=\frac{n-i}{(n+1)^i} \left[p\sum_{s=k}^{i} \binom{i}{s} (s+1)^{s-1} (n-s)^{i-s-1}\right. \notag \\
& \hspace{3cm} \left.+(1-p)\sum_{s=k}^{i} \binom{i}{s} (s+1)^{s-1} (n-s)^{i-s-1}\right] \notag \\
&=\frac{n-i}{(n+1)^i} \sum_{s=k}^i \binom{i}{s} (s+1)^{s-1} (n-s)^{i-s-1}.
\end{align*} 
The rest is immediate.
\end{proof}

From Proposition \ref{prob_displacement}, we see that the displacement
statistic is rather correlated: the displacement of a car depends on
the parking preferences of all the cars that enter before it. The next
corollary shows that there is more independence structure with the
unluckiness statistic. 

\begin{corollary}\label{cr:unlucky}
\begin{equation*}
\PR((i+1)\text{th car is unlucky} \hspace{.1cm} \vert \hspace{.1cm} \pi \in \PF(m, n))=\frac{i}{n+1}.
\end{equation*}
\end{corollary}

\begin{proof}
Note that a car is lucky if and only if its displacement is
zero. Using the notation of Proposition \ref{prob_displacement}, 
\begin{equation*}
\PR((i+1)\text{th car is lucky})=P_0\\=\frac{n-i}{(n+1)^i}
\sum_{s=0}^i \binom{i}{s} (s+1)^{s-1} (n-s)^{i-s-1}=1-\frac{i}{n+1}, 
\end{equation*}
where the last equality comes from Abel's extension of the binomial theorem.
\end{proof}

\begin{remark}
We will later give a direct proof of Corollary \ref{cr:unlucky} in the
more general case of $(r, k)$-parking functions. See Theorem
\ref{thm:general_k_r}. 
\end{remark}

The circular symmetry argument we have been employing is of wide
applicability and is not only restricted to classical parking
functions. It produces many new results and renders various known
results as corollaries. We now present some of them. 

Fix some positive integers $k$ and $r$. Let the length-$m$ vector
$\boldsymbol{u}$ be such that $u_i=k+(i-1)r$ for $1\leq i\leq m$. A
$\boldsymbol{u}$-parking 
function of this form is referred to in \cite{SW} as an $(r,
k)$-parking function of length $m$. There is a similar interpretation
for such $\boldsymbol{u}$-parking functions in terms of the classical parking
scenario: one wishes to park $m$ cars on a street with $u_m$ spots,
but only $m$ spots are still empty, which are at positions no later
than $u_1, \dots, u_m$. With the probabilistic scenario ($p$ not
necessarily equals $1$), a preference vector $\pi$ becomes a valid
parking function when the $i$th spot parked in by the $m$ cars read
from left to right is at most $u_i$ for $1\leq i\leq
m$. Denote the probability of this event by $\PR(\pi \in \PF_m(r,
k))$. When $p=1$, this event has probability either $0$ or $1$, and we
denote the deterministic set of valid parking functions also by
$\PF_m(r, k)$.

For example, take $k=r=m=2$ so $\boldsymbol{u}=(2, 4)$; then parking
preference $(3, 3) \in \PF_2(2, 2)$ with probability $1-p$. Here car
$1$ takes the empty spot $3$ upon entering the street. Now car $2$
enters and finds its desired spot $3$ taken by car $1$, prompting it
to make a choice. If car $2$ drives forward to park at spot $4$, which
happens with probability $p$, then the parking spots occupied by the
two cars are $(3, 4) \not\leq (2, 4)$. But if car $2$ drives backward
to park at spot $2$, which happens with probability $1-p$, then the
parking spots occupied by the two cars are $(2, 3) \leq (2, 4)$.

\begin{theorem}\label{thm:general_k_r}
\begin{equation}\label{general_k_r}
\sum_{\pi} x^{\unl(\pi)} y^{\rep(\pi)} \hspace{.1cm} \PR(\pi \in
\PF_m(r, k))=k \prod_{i=1}^{m-1} (xy+(i-1)x+k+mr-i).  
\end{equation}
\end{theorem}

\begin{proof}
We think deterministically and take the forward probability $p=1$. But
as will be clear in the proof, with minor adaptation our argument
works for a generic $p$. We utilize a generalization of Pollak's
original circle argument \cite{Pollak} for $(r, k)$-parking functions
introduced in \cite{Stanley1}. To generate an $(r, k)$-parking
function of length $m$ uniformly at random, we proceed as follows: 
\begin{enumerate}
\item Pick an element $\pi \in (\mathbb{Z}/(k+mr)\mathbb{Z})^m$, where
  the equivalence class representatives are taken in $1, \dots,
  k+mr$. 

\item For $i \in \{0, \dots, k+mr-1\}$, record $i$ if $\pi+i(1, \dots,
  1)$ (modulo $k+mr$) is an $(r, k)$-parking function, where $(1,
  \dots, 1)$ is a vector of length $m$. There should be exactly $k$
  such $i$'s. 

\item Pick one $i$ from (2) uniformly at random. Then $\pi+i(1, \dots,
  1)$ is an $(r, k)$-parking function of length $m$ taken uniformly at
  random. 
\end{enumerate}

The main takeaway from this procedure is that a random $(r,
k)$-parking function could be generated by assigning $m$ cars
independently on a circle of length $k+mr$ and then applying circular
rotation. Since circular rotation does not affect unluckiness nor
repeats, and whether the $(i+1)$th car is unlucky or repeats the
behavior of the $i$th car is independent of the particular preferences
of all the previous $i$ cars that have already parked, the generating
function factors completely. 

The first car is always lucky. For $1\leq i\leq m-1$, note that if the
$(i+1)$th car prefers any of the $i$ spots that are taken by the
previous $i$ cars, then it must be unlucky. In particular, if the
$(i+1)$th car repeats the preference of the $i$th car, then it must be
unlucky. We have 
$$\PR((i+1)\text{th car repeats the preference of } i\text{th car})=\frac{1}{k+mr},$$

$$\PR((i+1)\text{th car is unlucky but does not repeat the preference
  of } i\text{th car})=\frac{i-1}{k+mr},$$ 

$$\PR((i+1)\text{th car is lucky})=\frac{k+mr-i}{k+mr}.$$

Recall that $$\sum_{\pi} \PR(\pi \in \PF_m(r, k))=k(k+mr)^{m-1}.$$ This implies that
\begin{align*}
&\sum_{\pi} x^{\unl(\pi)} y^{\rep(\pi)} \hspace{.1cm} \PR(\pi \in
    \PF_m(r, k))\notag \\ 
=&k(k+mr)^{m-1} \prod_{i=1}^{m-1}
\left(xy\frac{1}{k+mr}+x\frac{i-1}{k+mr}+\frac{k+mr-i}{k+mr}\right)
\notag \\ 
=&k \prod_{i=1}^{m-1} (xy+(i-1)x+k+mr-i).
\end{align*}
\end{proof}

\begin{corollary}
Taking $k=n-m+1$ and $r=y=1$ in (\ref{general_k_r}), we have
\begin{equation*}
\sum_{\pi} x^{\unl(\pi)} \hspace{.1cm} \PR(\pi \in \PF(m, n)) =(n-m+1)
\prod_{i=1}^{m-1} (ix+(n-i+1)). 
\end{equation*}
This agrees with the corresponding (deterministic) formula in Gessel
and Seo \cite[Theorem 10.1, Corollary 10.2]{GS}. 
\end{corollary}

\begin{corollary}
Taking $k=r=x=1$ in (\ref{general_k_r}), we have
\begin{equation*}
\sum_{\pi} y^{\rep(\pi)} \hspace{.1cm} \PR(\pi \in \PF(m))=(y+m)^{m-1}.
\end{equation*}
This agrees with the corresponding (deterministic) formula in Yan
\cite[Corollary 1.3]{Yan}. 
\end{corollary}

A similar approach as in the proof of Theorem \ref{thm:general_k_r}
also yields the following result.

\begin{theorem}\label{thm:standalone}
\begin{equation*}
\sum_{\pi} x^{\unl(\pi)} y^{\lel(\pi)} \hspace{.1cm}
\PR(\pi \in \PF_m(r, k))=k y \prod_{i=1}^{m-1} (xy+(i-1)x+k+mr-i).  
\end{equation*}
\end{theorem}

\begin{proof}
The proof is almost identical to the proof of Theorem
\ref{thm:general_k_r}. We note that for $1\leq i\leq m-1$, if the
desired spot of the $(i+1)$th car coincides with that of the first
car, then it must be unlucky. 
\end{proof}

\begin{proposition}\label{followup}
\begin{equation*}
\sum_{\pi} x^{\unl(\pi)} y^{\#\{i\colon \pi_i=\pi_2\}} \hspace{.1cm}
\PR(\pi \in \PF_m(r, k))=k y \prod_{i=1}^{m-1} (xy+(i-1)x+k+mr-i).  
\end{equation*}
\end{proposition}

\begin{proof}
The above result will not come as a surprise once we note that an
instance of unluckiness occurs in the first two cars exactly when
$\pi_1=\pi_2$. Alternatively, we may switch the preferences of car $1$
and car $2$, which has no effect on the unluckiness index. 
\end{proof}

\begin{remark}
For $s \geq 3$, it is in general not true that
\begin{equation*}
\sum_{\pi} x^{\unl(\pi)} y^{\#\{i\colon \pi_i=\pi_s\}} \hspace{.1cm} \PR(\pi \in \PF_m(r, k))=k y \prod_{i=1}^{m-1} (xy+(i-1)x+k+mr-i). 
\end{equation*}
The first counterexample is located when $m=3$.
\end{remark}

\vskip.1truein

\begin{remark}\label{rk:symmetry}
It is however true that for $s \geq 1$,
\begin{equation*}
\sum_{\pi} y^{\#\{i\colon \pi_i=\pi_s\}} \hspace{.1cm} \PR(\pi \in
\PF_m(r, k))=k y \prod_{i=1}^{m-1} (y+(k+mr-1)).  
\end{equation*}
\end{remark}

Using standard probability tools, some asymptotic analysis of the
above parking statistics readily follows. 

\begin{proposition}\label{Poisson-CLT}
Take $m$ large. Take $r \geq 1$ any integer and $k=cm+r$ for some $c
\geq 0$. 
Consider the parking preference $\pi$ chosen uniformly at random. Let
$U(\pi)$ be the number of unlucky cars and $R(\pi)$ be the number of
repeats in $\pi$ reading from left to right. Then for fixed $j=0, 1,
\dots$, 
\begin{equation*}
\PR\left(R(\pi)=j \hspace{.1cm} \vert \hspace{.1cm} \pi \in \PF_m(r,
k) \right) \sim \frac{\left(\frac{1}{c+r}\right)^j
  e^{-\frac{1}{c+r}}}{j!}, 
\end{equation*}
and for fixed $x$, $-\infty<x<\infty$,
\begin{equation*}
\PR\left(\frac{U(\pi)-\frac{1}{2(c+r)}m}{\sqrt{\frac{3(c+r)-2}{6(c+r)^2}m}}\leq x \hspace{.1cm} \vert \hspace{.1cm} \pi \in \PF_m(r, k)\right) \sim \Phi(x),
\end{equation*}
where $\Phi(x)$ is the standard normal distribution function.
\end{proposition}


\begin{proof}
From the proof of Theorem \ref{thm:general_k_r}, the probability
generating function of $\boldsymbol{X}:=(U(\pi), R(\pi))$ may be
decomposed into $\boldsymbol{X}=\sum_{i=1}^{m-1}\boldsymbol{X}_i$, with
$\boldsymbol{X}_i=(U_i(\pi), R_i(\pi))$ independent, $U_i$ and $R_i$
both Bernoulli, and 
\begin{equation*}
\ER(U_i)=\frac{i}{k+mr}, \hspace{1cm} \ER(R_i)=\frac{1}{k+mr},
\end{equation*}
\begin{equation*}
\Var(U_i)=\frac{i}{k+mr}\left(1-\frac{i}{k+mr}\right), \hspace{1cm} \Var(R_i)=\frac{1}{k+mr}(1-\frac{1}{k+mr}).
\end{equation*}
We compute
\begin{equation*}
\sum_{i=1}^{m-1} \ER(U_i) \sim \frac{1}{2(c+r)}m, \hspace{1cm} \sum_{i=1}^{m-1} \ER(R_i) \sim \frac{1}{c+r},
\end{equation*}
\begin{equation*}
\sum_{i=1}^{m-1} \Var(U_i) \sim \frac{3(c+r)-2}{6(c+r)^2}m, \hspace{1cm} \sum_{i=1}^{m-1} \Var(R_i) \sim \frac{1}{c+r}. 
\end{equation*}
We recognize that $R(\pi)$ may be approximated by
Poisson$(\frac{1}{c+r})$ while $U(\pi)$ may be approximated by
normal$(0, 1)$ after standardization. 
\end{proof}

\begin{corollary}
Taking $c=0$ and $r=1$ (and thus $k=1$) in Proposition \ref{Poisson-CLT}, we have
\begin{equation*}
\PR\left(R(\pi)=j \hspace{.1cm} \vert \hspace{.1cm} \pi \in \PF(m) \right) \sim \frac{1}{ej!},
\end{equation*}
\begin{equation*}
\PR\left(\frac{U(\pi)-\frac{m}{2}}{\sqrt{\frac{m}{6}}}\leq x \hspace{.1cm} \vert \hspace{.1cm} \pi \in \PF(m)\right) \sim \Phi(x).
\end{equation*}
This agrees with the corresponding formula in Diaconis and Hicks \cite[Theorem 4, Theorem 6]{DH}.
\end{corollary}

\begin{proposition}\label{Poisson-CLT-2}
Take $m$ large. Take $1\leq s\leq m$ any integer. Take $r \geq 1$ any integer and $k=cm+r$ for some $c \geq 0$. 
Consider the parking preference $\pi$ chosen uniformly at random. Let
$L_s(\pi)$ be the number of cars with the same preference as car
$s$. Then for fixed $j=0, 1, \dots$, 
\begin{equation*}
\PR\left(L_s(\pi)=1+j \hspace{.1cm} \vert \hspace{.1cm} \pi \in \PF_m(r, k) \right) \sim \frac{\left(\frac{1}{c+r}\right)^j e^{-\frac{1}{c+r}}}{j!}.
\end{equation*}
\end{proposition}

\begin{proof}
This follows from Theorem \ref{thm:standalone} (see also Remark
\ref{rk:symmetry}). In the same way as the proof of Theorem
\ref{thm:standalone} is almost identical to the proof of Theorem
\ref{thm:general_k_r}, the proof of Proposition \ref{Poisson-CLT-2} is
also almost identical to the proof of Proposition \ref{Poisson-CLT}.
\end{proof}

\section{Circular Symmetry: Part II}

We have been focusing on the cars' perspective of parking functions
till now. In this section we also take into consideration the spots'
perspective. In the classical parking situation where there are $n$
cars parking on a street with $n$ spots, the unluckiness statistic
could be interpreted in two ways. Either it is the total number of
cars that fail to park at their desired spot, or it is the total number
of spots occupied by a car for which that spot is not the car's first
preference. In particular, spot $1$ is always lucky. We note that this
alternative interpretation for ``unluckiness'' fails for more general
parking situations. 

There is also a difference between cars and spots. Cars are actively
moving to find available spots; when a car's first preference is
taken, the car moves forward or backward and an extra probabilistic
scenario may be added to the parking model indicating the direction the
car moves. Contrarily, spots are passively waiting for cars to park in
them and a probabilistic scenario associated with the movement direction
is irrelevant. All statistics involving spots are thus only valid for
the deterministic parking model, where the forward probability $p=1$
(or equivalently $p=0$).

We note some curious features of the pair of statistics (leading
elements, 1's). While the leading elements statistic is invariant
under circular rotation, it does not satisfy permutation symmetry as
permuting the entries might change the first element. On the other
hand, though the 1's statistic is invariant under permuting all the
entries, it does not exhibit circular rotation invariance. Indeed,
only $1$ out of $n+1$ rotations of an assignment of $n$ cars on a
circle with $n+1$ spots gives a valid parking function.

We establish a recurrence relation for the pair of statistics
(unluckiness, leading elements) and (unluckiness, 1's). Write 
   $$ P_n(x, y)=\sum_{\pi \in \PF(n)} x^{\unl(\pi)} y^{\lel(\pi)} $$
and
  $$ Q_n(x, y)=\sum_{\pi \in \PF(n)} x^{\unl(\pi)} y^{\one(\pi)}. $$
Note that $P_n(x, y)$ is a special (deterministic) case of Theorem
\ref{thm:standalone} with $k=r=1$. As in the proof of Theorem
\ref{un1-general}, we focus on the last car. The parking protocol
implies that $(\pi_1, \dots, \pi_{n-1})$ may be decomposed into two
parking functions $\alpha \in \PF(i)$ and $\beta \in \PF(n-1-i)$,
where $0\leq i\leq n-1$, $\betaR$ is formed by subtracting $i+1$ from
the relevant entries in $\pi$, and $\alpha$ and $\beta$ do not
interact with each other. This open spot $i+1$ could be either the
same as $j$, the preference of the last car, in which case the car
parks directly; or $i+1$ could be bigger than $j$, in which case the
car travels forward to park. Also, depending on whether $\pi_1<i+1$ or
$\pi_1>i+1$ (note that it is impossible for $\pi_1=i+1$), the count on
$\lel(\pi)$ would be different, as $j$ has the possibility
of equalling $\pi_1$ in the former case but no possibility of
equalling $\pi_1$ in the latter case. \old{and
  $$R_n(x, y)=\sum_{\pi \in \PF(n)}
    x^{\unl(\pi)} y^{\one(\pi)} z^{\#\{i\colon  \pi_i=1\}}.$$} 

\begin{multline*}
P_{n}(x, y)=\sum_{i=0}^{n-1} \left(\binom{n-2}{i-1} (xy+(i-1)x+1) P_i(x, y) P_{n-i-1}(x, 1)\right.\\ \left.+\binom{n-2}{i} (ix+1) P_i(x, 1) P_{n-i-1}(x, y)\right).
\end{multline*}
 
\begin{equation*}
Q_{n}(x, y)=\sum_{i=1}^{n-1} \binom{n-1}{i} (xy+(i-1)x+1) Q_i(x, y) Q_{n-i-1}(x, 1)+yQ_{n-1}(x, 1).
\end{equation*}

The recurrence formulas for $P_n(x, y)$ and $Q_n(x, y)$ look very
different, yet we will show that $P_n(x, y)=Q_n(x, y)$ in Theorem
\ref{thm:all}. We first present a direct combinatorial argument for
counting the number of parking functions $\pi \in \PF(n)$ with
$\pi_1=1$. This will shed light on the structure of parking functions
and will be useful in the proof of Theorem \ref{thm:all}. 
\begin{lemma}\label{simple}
  We have
  \begin{equation*}
      \# \{\pi \in \PF(n): \pi_1=1\} =2(n+1)^{n-2}.    
\end{equation*}
\end{lemma}

\begin{proof}
We assign cars $2, \dots, n$ independently on a circle of length
$n+1$, of which there are $(n+1)^{n-1}$ possibilities. This leaves two
empty spots on the circle. Now in order for car $1$'s preference to be
recorded as $1$ after circular rotation, car $1$ has to choose the
adjacent spot (clockwise) to either of the empty spots, so only two
possibilities. See Figure \ref{car1:illustration}. Note that a valid
parking function is produced when spot $n+1$ is left unoccupied after
circular rotation, and one $n+1$ scalar factor goes away. 

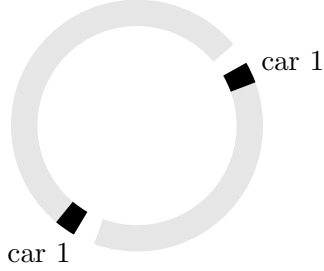
\begin{figure}
\begin{tikzpicture}[scale=1.5]
\draw[color=gray!20, line width=10pt] (40:1) arc (40:240:1);
\draw[color=gray!20, line width=10pt] (250:1) arc (250:390:1);

\draw[color=black, line width=10pt] (230:1) arc (230:240:1);
\draw[color=black, line width=10pt] (380:1) arc (380:390:1);

\node[anchor=north east] at (-0.5, -0.95) {car $1$};
\node[anchor=north east] at (1.75, 0.75) {car $1$};
\end{tikzpicture}
\caption{Possible positions of car 1 on the circle.}
\label{car1:illustration}
\end{figure}
\end{proof}

\begin{theorem}\label{thm:all}
\begin{align}\label{master}
&\sum_{\pi \in \PF(n)} x^{\lel(\pi)} y^{\one(\pi)}
  z^{\unl(\pi)} \notag \\ 
&=xy\left((n-1)\prod_{i=1}^{n-2} (xz+yz+(i-1)z+n-i)+(xyz+1)\prod_{i=1}^{n-2} (xyz+iz+n-i)\right).
\end{align}
\end{theorem}

\begin{proof}
We classify into two situations: $\pi_1 \neq 1$ and $\pi_1=1$.

Case A: $\pi_1 \neq 1$. As usual, we assign $n$ cars independently on
a circle of length $n+1$ and then apply circular rotation. Recall that
for classical parking functions, spot $1$ is always taken by some
car. Since $\pi_1 \neq 1$, there are two special spots on the circle,
one corresponding to $\pi_1$ and the other corresponding to $1$ (both
pre-rotation). These two special spots must correspond to lucky
cars. We could think that the first car that picks spot $1$ enters the
street right after the leading car that prefers spot $\pi_1$ and ahead
of all other cars, and this will have no effect on either the
unluckiness, or the leading elements statistic, or the 1's
statistic. When a new car comes in and picks the special spot which
will rotate to spot $\pi_1$, it must be unlucky. This situation adds
$1$ to the leading elements statistic. Similarly, when the new car
picks the special spot which will rotate to spot $1$ upon entering, it
must also be unlucky. This situation adds $1$ to the 1's
statistic. Lastly, if the new car picks any of the already taken
spots, it must be unlucky as well. Suppose $i$ spots (including the
two special spots corresponding to $\pi_1$ and $1$) have been taken,
where $2\leq i \leq n-1$, then with the entrance of the new car, we
have 
$$\PR(\text{new car picks spot 1})=\frac{1}{n+1},$$

$$\PR(\text{new car picks spot } \pi_1)=\frac{1}{n+1},$$

$$\PR(\text{new car picks an unavailable spot that is neither spot 1 nor spot } \pi_1)=\frac{i-2}{n+1},$$

$$\PR(\text{new car picks an available spot})=\frac{n-i+1}{n+1}.$$
From Lemma \ref{simple}, the total number of parking functions $\pi \in \PF(n)$ where $\pi_1 \neq 1$ is given by
$$(n+1)^{n-1}-2(n+1)^{n-2}=(n+1)^{n-2}(n-1).$$ This implies that
\begin{align*}
&\sum_{\pi \in \PF(n)\colon \pi_1 \neq 1} x^{\lel(\pi)}
  y^{\one(\pi)} z^{\unl(\pi)} \notag \\  
=&(n+1)^{n-2}(n-1) xy\prod_{i=2}^{n-1}
\left(xz\frac{1}{n+1}+yz\frac{1}{n+1}+z\frac{i-2}{n+1}+\frac{n-i+1}{n+1}\right)
\notag \\ 
=&xy(n-1)\prod_{i=1}^{n-2} (xz+yz+(i-1)z+n-i).
\end{align*}

Case B: $\pi_1=1$. Since $\pi_1=1$, there is only one special spot on
the circle, corresponding to $1$ pre-rotation. This special spot must
correspond to a lucky car. But from the proof of Lemma \ref{simple},
we see that there is some other car that is also special. It either
parks in the adjacent-to-empty spot that is picked by car $1$ (thus
repeating the preference of car $1$) or it parks in the other
adjacent-to-empty spot (thus lucky). The first situation contributes
$xyz$ to the generating function, while the second situation
contributes $1$ to the generating function. Let us now focus on the
other $n-2$ cars. Suppose $i$ spots (including the special spot
corresponding to $1$ and the other special car) have been taken, where
$2\leq i \leq n-1$, then with the entrance of the new car, we have 
  $$\PR(\text{new car picks spot 1})=\frac{1}{n+1},$$

$$\PR(\text{new car picks an unavailable spot that is not spot 1})=\frac{i-1}{n+1},$$

$$\PR(\text{new car picks an available spot})=\frac{n-i+1}{n+1}.$$
Again from the proof of Lemma \ref{simple}, the total number of
parking preferences for these non-special $n-2$ cars is given by
$(n+1)^{n-2}$. This implies that 
\begin{align*}
&\sum_{\pi \in \PF(n)\colon\pi_1=1} x^{\lel(\pi)} y^{\one(\pi)}
  z^{\unl(\pi)} \notag \\ 
=&(n+1)^{n-2}xy(xyz+1)\prod_{i=2}^{n-1} \left(xyz\frac{1}{n+1}+z\frac{i-1}{n+1}+\frac{n-i+1}{n+1}\right) \notag \\
=&xy(xyz+1)\prod_{i=1}^{n-2} (xyz+iz+n-i).
\end{align*}
\end{proof}

\begin{corollary}
Taking $x=1$ and $z=1$ in (\ref{master}), we have
\begin{equation}\label{ch}
\sum_{\pi \in \PF(n)} y^{\one(\pi)}=y(y+n)^{n-1}.
\end{equation}
This agrees with the corresponding (deterministic) formula in Yan
\cite[Corollary 1.16]{Yan}. The coefficients count forests of rooted
trees with a specified number of components (trees). We note that
(\ref{ch}) is also the characteristic polynomial (after replacing $y$
by $-y$ and multiplying by $(-1)^n$) of both the Shi arrangement and
the Ish arrangement. See Stanley \cite[Theorem 5.16]{Stanley2} and
Armstrong and Rhoades \cite[Theorem 3.2]{AR}. 
\end{corollary}

\begin{remark}
Using standard probability tools, some asymptotic analysis for the
$1$'s parking statistic was done in Diaconis and Hicks \cite[Theorem
  7]{DH}. 
\end{remark}

Setting $z=1$ in our ``master formula'' (\ref{master}), we may use the
multinomial theorem to get a simple explicit formula for the
coefficients. We will present a direct combinatorial proof for the
formula. The following structural result about multisets will be
useful in our proof. 

\begin{lemma}\label{lm:structure}
Suppose for some $k\geq 1$, elements $1, 2, \dots, k$ all appear in a
parking function $\pi \in \PF(n)$. Fix the elements in $\pi$ that
are equal to $1, 2, \dots, k$ (denote the total number of such
elements by $s$) and rotate the remaining elements of $\pi$ in the
cycle $(k+1, k+2, \dots, n, n+1)$. If $s<n$, then exactly $s-k+1$ of
these are parking functions. 
\end{lemma}

\begin{proof}
We create an extra parking spot $n+1$ and let the cars go around in a
circle if necessary. By a slight generalization of Pollak's original
argument \cite{Pollak}, we note that when there are $l$ cars parking
on a circle with $m$ spots, exactly $m-l$ of the rotations of a fixed
parking preference will leave spot $n+1$ unoccupied and are thus valid
parking functions. Now spots $1, 2, \dots, k$ are all taken and the
parking situation is equivalent to parking $n-s$ cars on a circle with
$n+1-k$ spots, so $(n+1-k)-(n-s)=s-k+1$ rotations will give valid
parking functions. 
\end{proof}

\begin{proposition}\label{thm:master}
  We have for $s \neq t$,
\begin{equation*} 
\#\{\pi \in \PF(n)\colon  \lel(\pi)=s+1 \text{ and } \one(\pi)=t+1\}=\binom{n-2}{s, t, n-s-t-2} (n-1)^{n-s-t-1},    
\end{equation*}
and for $s=t$,
\begin{multline*}
\#\{\pi \in \PF(n)\colon  \lel(\pi)=s+1 \text{ and } \one(\pi)=s+1\}\\=\binom{n-2}{s, s, n-2s-2} (n-1)^{n-2s-1}+\binom{n-2}{s} n^{n-s-2}+\binom{n-2}{s-1} n^{n-s-1}.    
\end{multline*}
\end{proposition}

\begin{proof}
We use a variant of Pollak's circle argument \cite{Pollak}. We
classify into two situations: $\pi_1 \neq 1$ and $\pi_1=1$. 

Case A: $\pi_1 \neq 1$. Car $1$'s parking preference is already fixed
to be $\pi_1$, but we have to choose the parking preferences for the
other $n-1$ cars. This contributes a multinomial coefficient to the
list of preferences $\pi$. Note that parking functions are invariant
under the action of the symmetric group $\Sym_n$ permuting the $n$
cars. We create an extra parking spot $n+1$ and let the cars go around
in a circle if necessary. Since cars with preference $1$ will never
back up to spot $n+1$, we could think that all the cars that prefer
spot $1$ enter the street right after the leading car that prefers
spot $\pi_1$ and ahead of all other cars. We now focus attention on
these trailing cars, and note that $s$ of these cars prefer spot
$\pi_1$ (which has $n-1$ possible values since $\pi_1 \neq 1$), and
each of the remaining $n-s-t-2$ cars could choose any spot from $\{1,
\dots, n+1\} \setminus \{1, \pi_1\}$. The parking functions among
these are exactly those for which spot $n+1$ remains unoccupied. Since
there are $n-t-2$ cars in total and $n-1$ spots, this is a fraction
$(t+1)/(n-1)$ of these functions, as explained in Lemma
\ref{lm:structure}. Thus we obtain 
  $$\frac{t+1}{n-1}(n-1)(n-1)^{n-s-t-2}=(t+1)(n-1)^{n-s-t-2}$$
valid preferences from these later cars. Putting all this together,
\begin{align*}
&\#\{\pi \in \PF(n)\colon  \pi_1 \neq 1, \hspace{.1cm} \lel(\pi)=s+1 \text{ and } \one(\pi)=t+1\}\notag \\=&\binom{n-1}{s, t+1, n-s-t-2} (t+1)(n-1)^{n-s-t-2}=\binom{n-2}{s, t, n-s-t-2} (n-1)^{n-s-t-1}.
\end{align*}

Case B: $\pi_1=1$. The argument works similarly but is easier. Car
$1$'s parking preference is fixed to be $1$ so $s=t$. We choose the
parking preferences for the other $n-1$ cars. This contributes a
multinomial (actually, binomial in this case) coefficient to the list
of preferences $\pi$. Again we could think that all the cars that
prefer spot $1$ enter the street before all the other cars. Focusing
attention on these trailing cars, we note that each of the $n-s-1$
cars could choose any spot from $\{2, \dots, n+1\}$. The parking
functions among these are exactly those for which spot $n+1$ remains
unoccupied. Since there are $n-s-1$ cars in total and $n$ spots, this
is a fraction $(s+1)/n$ of these functions, as explained in Lemma
\ref{lm:structure}. Thus we obtain 
  $$\frac{s+1}{n}n^{n-s-1}=(s+1)n^{n-s-2}$$
valid preferences from these later cars. Putting all this together,
\begin{align*}
&\#\{\pi \in \PF(n)\colon  \pi_1=1, \hspace{.1cm} \lel(\pi)=s+1 \text{ and }
  \one(\pi)=s+1\}\notag \\=&\binom{n-1}{s} (s+1)n^{n-s-2} 
=\binom{n-2}{s} n^{n-s-2}+\binom{n-2}{s-1} n^{n-s-1}.
\end{align*}

Finally, we note that the $s\neq t$ term could only come from Case A,
whereas the $s=t$ term could come from either Case A or Case B. 
\end{proof}

From Theorems \ref{thm:general_k_r} and \ref{thm:all}, we recognize
that the number of classical parking functions of length $n$ with $k$
repeats coincides with the number of parking functions with $k+1$
elements equal to the leading element and also coincides with the
number of parking functions where there are $k+1$ $1$'s appearing in
the parking preference sequence. See Table
\ref{tree-parking:illustration} for the case where $n=3$. Here the
first two columns are connected by the correspondence between labelled
trees and their Pr\"{u}fer code. The number of $0$'s in the Pr\"{u}fer
code plus $1$ is the same as the number of children of the root vertex
$0$. The bijection between the first and third columns is due to
Pollak \cite{Pollak2}. The number of $0$'s in the Pr\"{u}fer code is
the same as the number of repeats in the parking function. The fourth
column arises from the Pr\"{u}fer code via circular symmetry. We add a
$0$ in front of the Pr\"{u}fer code and interpret $0$ as spot $n+1$ on
the circle. Then there is exactly one rotation of the circle that
gives a valid parking function. The number of $0$'s in the Pr\"{u}fer
code plus $1$ is the same as the number of times the leading element
appears in the parking function. The fifth column is generated from
the labelled tree using a breadth first search as described in Yan
\cite[Section 1.2.3]{Yan}. The number of children of the root vertex
$0$ is the same as the number of $1$'s in the parking preference
sequence. 

We end this section with some observations. Our first observation
concerns corresponding statistics between parking functions and
trees. We note that even though we have the generating function
equality for individual parking statistics, 
 $$\sum_{\pi \in \PF(n)} y^{\lel(\pi)}=\sum_{\pi \in \PF(n)}
  y^{\one(\pi)}=\sum_{T \in \mathcal{T}(n+1)}
  y^{\deg_T(0)}=y(y+n)^{n-1} $$ 
  and
   $$\sum_{\pi \in \PF(n)} x^{\unl(\pi)}=\sum_{T \in \mathcal{T}(n+1)}
      x^{\nld(T)}=\prod_{i=1}^{n-1} (ix+(n-i+1)),$$ 
the equality nevertheless fails for pairs:
\begin{equation*}
\sum_{\pi \in \PF(n)} x^{\unl(\pi)} y^{\lel(\pi)}=\sum_{\pi \in
  \PF(n)} x^{\unl(\pi)} y^{\one(\pi)} \\ \neq \sum_{T \in
  \mathcal{T}(n+1)} x^{\nld(T)} y^{\deg_T(0)}. 
\end{equation*}
Indeed, from Theorem \ref{thm:all}, for parking functions
\begin{equation*}
\sum_{\pi \in \PF(n)} x^{\unl(\pi)} y^{\lel(\pi)}=\sum_{\pi \in
  \PF(n)} x^{\unl(\pi)} y^{\one(\pi)}\\ =y\prod_{i=1}^{n-1}
     (xy+(i-1)x+(n-i+1)),  
\end{equation*}
while for trees, the counterpart formula was derived in Gessel and Seo
\cite[Theorem 6.1]{GS}, which gives 
\begin{equation*}
  \sum_{T \in \mathcal{T}(n+1)} x^{\nld(T)}
  y^{\deg_T(0)}=y\prod_{i=1}^{n-1} (ix+(n-i+y)).  
\end{equation*}
The first counterexample is located when $n=2$. (See Hou \cite{Hou}
for more similar but contrasting generating functions for statistics
in trees.)

Similarly, we also note that even though
$$\sum_{\pi \in \PF(n)} y^{\dis(\pi)}=\sum_{T \in \mathcal{T}(n+1)}
y^{\inv(T)} \quad \text{ and } \quad \sum_{\pi \in \PF(n)} x^{\unl(\pi)}=\sum_{T
  \in \mathcal{T}(n+1)} x^{\edes(T)},$$ 
we have
$$\sum_{\pi \in \PF(n)} x^{\unl(\pi)} y^{\dis(\pi)} \neq \sum_{T \in
  \mathcal{T}(n+1)} x^{\edes(T)} y^{\inv(T)}.$$ 
The first counterexample is located when $n=4$.

Our second observation examines statistics of parking functions
alone. Borrowing notation $\lel(\pi)$ from ``leading elements'', we write $\nlel(\pi)$ for the
total number of cars whose desired spot is the same as that of the second
car. From Theorem \ref{thm:standalone}, Proposition \ref{followup}
and the previous observation, we have 
$$\sum_{\pi \in \PF(n)} x^{\unl(\pi)} y^{\lel(\pi)}=\sum_{\pi \in
  \PF(n)} x^{\unl(\pi)} y^{\nlel(\pi)}=\sum_{\pi \in
  \PF(n)} x^{\unl(\pi)} y^{\one(\pi)}.$$ 
Comparing with Theorem \ref{thm:all}, we will show that the generating
functions for the triple $$(\nlel(\pi), \one(\pi), \unl(\pi)) \quad  \mathrm{and} \quad
(\lel(\pi), \one(\pi), \unl(\pi))$$ are
the same, but the generating function for the triple $(\lel(\pi),
\nlel(\pi), \unl(\pi))$ is different. 

\begin{theorem}\label{thm:correpondence}
\begin{align*}
&\sum_{\pi \in \PF(n)} x^{\nlel(\pi)} y^{\one(\pi)}
  z^{\unl(\pi)} \notag \\  
&=xy\left((n-1)\prod_{i=1}^{n-2} (xz+yz+(i-1)z+n-i)+(xyz+1)\prod_{i=1}^{n-2} (xyz+iz+n-i)\right).
\end{align*}
\end{theorem}

\begin{proof}
As in the proof of Proposition \ref{followup}, we switch the
preferences of car $1$ and car $2$, which has no effect on the
unluckiness or the $1$'s statistic. The result readily follows from
Theorem \ref{thm:all}. 
\end{proof}

We now present a direct combinatorial argument for counting the number
of parking functions $\pi \in \PF(n)$ with $\pi_1=\pi_2$ (actually, in
the more general case $\pi_1=\cdots=\pi_k$ for some $k$). This will
shed light on the structure of parking functions and will be useful in
the proof of Theorem \ref{thm:differ}. Note that the parking
statistics in Lemma \ref{simple2} and Theorem \ref{thm:differ} 
involve only car preferences and not spots, so the statements could be
interpreted probabilistically as in Section \ref{sec:sym1}. 

\begin{lemma}\label{simple2}
Let $k\geq 1$. 
\begin{equation*}
\#\{\pi \in \PF(n)\colon \pi_1=\pi_2=\cdots=\pi_k\}=(n+1)^{n-k}.    
\end{equation*}
\end{lemma}

\begin{proof}
We apply circular symmetry. Add an additional space $n+1$ and arrange
the spaces in a circle. We select a common spot for the first $k$
cars, and it can be done in $n+1$ ways. Then for the remaining $n-k$
cars, there are $(n+1)^{n-k}$ possible preference sequences. A valid
parking function is produced when spot $n+1$ is left unoccupied after
circular rotation, and one $n+1$ scalar factor goes away. 
\end{proof}

\begin{theorem}\label{thm:differ}
\begin{align}\label{contrast}
&\sum_{\pi \in \PF(n)} x^{\lel(\pi)} y^{\nlel(\pi)} z^{\unl(\pi)} \notag \\
&=xy\left(n\prod_{i=1}^{n-2} (xz+yz+(i-1)z+n-i)+xyz\prod_{i=1}^{n-2} (xyz+iz+n-i)\right).
\end{align}
\end{theorem}

\begin{proof}
We proceed as in the proof of Theorem \ref{thm:all}. We classify into two situations: $\pi_1 \neq \pi_2$ and $\pi_1=\pi_2$.

Case A: $\pi_1 \neq \pi_2$. As usual, we assign $n$ cars independently
on a circle of length $n+1$ and then apply circular rotation. Since
$\pi_1 \neq \pi_2$, there are two special spots on the circle, one
corresponding to $\pi_1$ and the other corresponding to $\pi_2$ (both
pre-rotation). These two special spots must correspond to lucky
cars. When a new car comes in and picks the special spot which will
rotate to spot $\pi_1$, it must be unlucky. Similarly, when the new
car picks the special spot which will rotate to spot $\pi_2$ upon
entering, it must also be unlucky. Lastly, if the new car picks any of
the already taken spots, it must be unlucky as well. Suppose $i$ spots
(including the two special spots corresponding to $\pi_1$ and $\pi_2$)
have been taken, where $2\leq i \leq n-1$, then with the entrance of
the new car, we have 
$$\PR(\text{new car picks spot } \pi_1)=\frac{1}{n+1},$$

$$\PR(\text{new car picks spot } \pi_2)=\frac{1}{n+1},$$

$$\PR(\text{new car picks an unavailable spot that is neither spot } 
\pi_1 \text{ nor spot } \pi_2)=\frac{i-2}{n+1},$$

$$\PR(\text{new car picks an available spot})=\frac{n-i+1}{n+1}.$$
From Lemma \ref{simple2}, the total number of parking functions $\pi \in \PF(n)$ where $\pi_1 \neq \pi_2$ is given by
$$(n+1)^{n-1}-(n+1)^{n-2}=n(n+1)^{n-2}.$$ This implies that
\begin{align*}
&\sum_{\pi \in \PF(n)\colon\pi_1 \neq \pi_2} x^{\lel(\pi)} y^{\nlel(\pi)} z^{\unl(\pi)} \notag \\
=&n(n+1)^{n-2} xy\prod_{i=2}^{n-1} \left(xz\frac{1}{n+1}+yz\frac{1}{n+1}+z\frac{i-2}{n+1}+\frac{n-i+1}{n+1}\right) \notag \\
=&nxy\prod_{i=1}^{n-2} (xz+yz+(i-1)z+n-i).
\end{align*}

Case B: $\pi_1=\pi_2$. Since $\pi_1=\pi_2$, $\pi_2$ parks in an
adjacent spot (clockwise) to $\pi_1$. Thus car $1$ is lucky and car
$2$ is unlucky. Let us turn our attention to the remaining $n-2$
cars. Suppose $i$ spots (including the two special spots corresponding
to $\pi_1$ and $\pi_1+1$) have been taken, where $2\leq i \leq n-1$,
then with the entrance of the new car, we have 
$$\PR(\text{new car picks spot } \pi_1)=\frac{1}{n+1},$$

$$\PR(\text{new car picks an unavailable spot that is not spot } \pi_1)=\frac{i-1}{n+1},$$

$$\PR(\text{new car picks an available spot})=\frac{n-i+1}{n+1}.$$
Again from the proof of Lemma \ref{simple2}, the total number of parking preferences for these remaining $n-2$ cars is given by $(n+1)^{n-2}$. This implies that
\begin{align*}
&\sum_{\pi \in \PF(n)\colon\pi_1=\pi_2} x^{\lel(\pi)} y^{\nlel(\pi)} z^{\unl(\pi)} \notag \\
=&(n+1)^{n-2}xy(xyz)\prod_{i=2}^{n-1} \left(xyz\frac{1}{n+1}+z\frac{i-1}{n+1}+\frac{n-i+1}{n+1}\right) \notag \\
=&xy(xyz)\prod_{i=1}^{n-2} (xyz+iz+n-i).
\end{align*}
\end{proof}

Setting $z=1$ in our ``contrast formula'' (\ref{contrast}), as with Proposition \ref{thm:master}, a direct combinatorial proof is straightforward.

\begin{proposition}\label{thm:master2}
We have for $s \neq t$,
\begin{equation*}
\#\{\pi \in \PF(n)\colon  \lel(\pi)=s+1 \text{ and } \nlel(\pi)=t+1\}=\binom{n-2}{s, t, n-s-t-2} n(n-1)^{n-s-t-2},    
\end{equation*}
and for $s=t$,
\begin{multline*}
\#\{\pi \in \PF(n)\colon  \lel(\pi)=s+1 \text{ and } \nlel(\pi)=s+1\}\\=\binom{n-2}{s, s, n-2s-2} n(n-1)^{n-2s-2}+\binom{n-2}{s-1} n^{n-s-1}.    
\end{multline*}
\end{proposition}

\begin{proof}
We apply Pollak's circle argument \cite{Pollak} and classify into two
situations: $\pi_1 \neq \pi_2$ and $\pi_1=\pi_2$. The argument is
analogous but easier than the proof of Proposition \ref{thm:master}. 

Case A: $\pi_1 \neq \pi_2$. Create an extra parking spot $n+1$ and let
the cars go around in a circle if necessary. We select two spots and
assign them respectively to car $1$ and car $2$, of which there are
$(n+1)n$ possibilities. Next we choose the parking preferences for the
other $n-2$ cars. Now $s$ of these cars have the same preference as car
$1$ and $t$ of these cars have the same preference as car $2$, and
this contributes a multinomial coefficient to the list of preferences
$\pi$. Lastly, the remaining $n-s-t-2$ cars could choose any of the
$n-1$ spots not preferred by car $1$ and car $2$, and there are
$(n-1)^{n-s-t-2}$ possible preference sequences. A valid parking
function is produced when spot $n+1$ is left unoccupied after circular
rotation, and one $n+1$ scalar factor goes away. Putting all this
together, 
\begin{equation*}
\#\{\pi \in \PF(n)\colon   \pi_1 \neq \pi_2, \hspace{.1cm} \lel(\pi)=s+1 \text{ and } \nlel(\pi)=t+1\}=\binom{n-2}{s, t, n-s-t-2} n(n-1)^{n-s-t-2}.
\end{equation*}

Case B: $\pi_1=\pi_2$. As in Case A, we assign a common spot on the
circle to car $1$ and car $2$, of which there are $n+1$
possibilities. Next we choose the parking preferences for the other
$n-2$ cars. Now $s-1$ of these cars have the same preference as car $1$,
and this contributes a multinomial (actually, binomial in this case)
coefficient to the list of preferences $\pi$. Lastly, the remaining
$n-s-1$ cars could choose any of the $n$ spots not preferred by car
$1$, and there are $n^{n-s-1}$ possible preference sequences. A valid
parking function is produced when spot $n+1$ is left unoccupied after
circular rotation, and one $n+1$ scalar factor goes away. Putting all
this together, 
\begin{equation*}
\#\{\pi \in \PF(n)\colon  \pi_1=\pi_2, \hspace{.1cm} \lel(\pi)=s+1 \text{ and
} \nlel(\pi)=s+1\}=\binom{n-2}{s-1} n^{n-s-1}. 
\end{equation*}

Finally, we note that the $s\neq t$ term could only come from Case A,
whereas the $s=t$ term could come from either Case A or Case B. 
\end{proof}

\old{\textcolor{red}{Another alternate proof.} Following our
  derivations, the number of parking functions where the 1's statistic
  is $k$ and the number of parking functions where the leading
  statistic is $k$ are both given by $\binom{n-1}{k-1} n^{n-k}$. A
  combinatorial proof for the former is well-known, and here we
  present a combinatorial proof for the latter. 

\begin{corollary}
\begin{equation*}
\sum_{\pi \in \PF(n)} \#\{\pi\colon  \one(\pi)=k\}=\binom{n-1}{k-1} n^{n-k}.    
\end{equation*}
\end{corollary}

\begin{proof}
We again apply circular symmetry. Add an additional space $n+1$ and
arrange the spaces in a circle. Choosing the $k-1$ cars out of the
trailing cars to have the same preference as the leading car
contributes a binomial coefficient. We first select a common spot for
these cars, and it can be done in $n+1$ ways. Then for the remaining
$n-k$ cars, there are $n^{n-k}$ possible preference sequences. A valid
parking function is produced when spot $n+1$ is left unoccupied after
circular rotation, and the $n+1$ scalar factor goes away. 
\end{proof}}

\section{Miscellaneous results}

As mentioned earlier, the techniques we have been using (recurrence,
circular symmetry, $\dots$) in this paper are not restricted to
classical parking functions. We have already seen some applicability
of our method to $(r, k)$-parking functions in Section
\ref{sec:sym1}. In this section we will further demonstrate the power
of our systematic approach in the study of parking functions through
some selected results. 

\subsection{Prime parking functions} A classical parking function
$\pi=(\pi_1, \dots, \pi_n)$ is said to be prime if for all $1 \leq j
\leq n-1$, at least $j+1$ cars want to park in the first $j$
places. (Equivalently, if we remove some term of $\pi$ equal to 1,
then we still have a parking function.) Denote the set of prime
parking functions of length $n$ by $\PPF(n)$. As with classical
parking functions, we could also study prime parking functions via
circular rotation \cite[pp.~141-142]{Stanley}. We note however that
the circular symmetry argument for prime parking functions arises
somewhat differently from the corresponding argument for classical
parking functions, and only the deterministic parking protocol works
for this prime parking model. 
\begin{theorem}\label{thm:PPF}
We have
\begin{equation}\label{ppf_ur}
\sum_{\pi \in \PPF(n)} x^{\unl(\pi)} y^{\rep(\pi)}=\prod_{i=1}^{n-1}
(xy+(i-1)x+n-1-i). 
\end{equation}
\end{theorem}

\begin{proof}
The proof is similar to the proof of Theorem \ref{thm:general_k_r}. We
utilize the circular symmetry argument of Kalikow \cite{Stanley} and
interpret in terms of probability. 
\begin{enumerate}
\item Pick an element $\pi \in (\mathbb{Z}/(n-1)\mathbb{Z})^n$, where the equivalence class representatives are taken in $1, \dots, n-1$.

\item For $i \in \{0, \dots, n-2\}$, record $i$ if $\pi+i(1, \dots,
  1)$ (modulo $n-1$) is a prime parking function, where $(1, \dots,
  1)$ is a vector of length $n$. There is exactly one such $i$, which
  we denote by $i_0$. Then $\pi+i_0(1, \dots, 1)$ is a prime parking
  function of length $n$ taken uniformly at random. 
\end{enumerate}

The main takeaway from this procedure is that a random prime parking
function could be generated by assigning $n$ cars independently on a
circle of length $n-1$ and then applying circular rotation. As in the
proof of Theorem \ref{thm:general_k_r}, the generating function
factors completely. 

The first car is always lucky. For $1\leq i\leq n-1$, note that if the
$(i+1)$th car prefers any of the $i$ spots that are taken by the
previous $i$ cars, then it must be unlucky. In particular, if the
$(i+1)$th car repeats the preference of the $i$th car, then it must be
unlucky. The last car is always unlucky. We have 
$$\PR((i+1)\text{th car repeats the preference of } i\text{th car})=\frac{1}{n-1},$$

  $$\PR((i+1)\text{th car is unlucky but does not repeat the preference
  of } i\text{th car})=\frac{i-1}{n-1},$$ 

$$\PR((i+1)\text{th car is lucky})=\frac{n-1-i}{n-1}.$$

Since the total number of prime parking functions of length $n$ is
$(n-1)^{n-1}$, this implies that 
\begin{align*}
\sum_{\pi \in \PPF(n)} x^{\unl(\pi)} y^{\rep(\pi)}=&(n-1)^{n-1}
\prod_{i=1}^{n-1}
\left(xy\frac{1}{n-1}+x\frac{i-1}{n-1}+\frac{n-1-i}{n-1}\right) \notag
\\ 
=&\prod_{i=1}^{n-1} (xy+(i-1)x+n-1-i).
\end{align*}
\end{proof}

\begin{corollary}
Taking $y=1$ in (\ref{ppf_ur}), we have
\begin{equation*}
\sum_{\pi \in \PPF(n)} x^{\unl(\pi)}=\prod_{i=1}^{n-1}
(ix+(n-1-i)). 
\end{equation*}
This agrees with the corresponding formula in Gessel and Seo \cite[Corollary 10.3]{GS}.
\end{corollary}

\begin{corollary}
Taking $x=1$ in (\ref{ppf_ur}), we have
\begin{equation*}
\sum_{\pi \in \PPF(n)} y^{\rep(\pi)}=(y+n-2)^{n-1}.
\end{equation*}
This agrees with the corresponding formula in Kalikow \cite[Section 1.7.1]{Kalikow}.
\end{corollary}

\begin{theorem}\label{thm:PPF-2}
\begin{equation*}
\sum_{\pi \in \PPF(n)} x^{\unl(\pi)} y^{\lel(\pi)}=y\prod_{i=1}^{n-1}
(xy+(i-1)x+n-1-i). 
\end{equation*}
\end{theorem}

\begin{proof}
The proof is almost identical to the proof of Theorem
\ref{thm:PPF}. We note that for $1\leq i\leq n-1$, if the desired spot
of the $(i+1)$th car coincides with that of the first car, then it
must be unlucky. 
\end{proof}

\begin{remark}\label{rk:PPF-2}
As in classical parking functions, prime parking functions are
invariant under the action of the symmetric group $\Sym_n$ permuting
the $n$ cars. Thus for $s\geq 1$, 
\begin{equation*}
\sum_{\pi \in \PPF(n)} y^{\#\{i\colon  \pi_i=\pi_s\}}=y(y+n-2)^{n-1}.
\end{equation*}
\end{remark}

Some asymptotic analysis of the above parking statistics readily follows.

\begin{proposition}\label{Poisson-CLT-PPF}
Take $n$ large and $1\leq s\leq n$ any integer. Consider parking
preference $\pi$ chosen uniformly at random. Let $U(\pi)$ be the
number of unlucky cars, $R(\pi)$ be the number of repeats in $\pi$
reading from left to right, and $L_s(\pi)$ be the number of cars with
the same preference as car $s$. Then for fixed $j=0, 1, \dots$, 
\begin{equation*}
\PR\left(R(\pi)=j \hspace{.1cm} \vert \hspace{.1cm} \pi \in \PPF(n) \right) \sim \frac{1}{ej!},
\end{equation*}
\begin{equation*}
\PR\left(L_s(\pi)=1+j \hspace{.1cm} \vert \hspace{.1cm} \pi \in \PPF(n) \right) \sim \frac{1}{ej!},
\end{equation*}
and for fixed $x$, $-\infty<x<\infty$,
\begin{equation*}
\PR\left(\frac{U(\pi)-\frac{n}{2}}{\sqrt{\frac{n}{6}}}\leq x \hspace{.1cm} \vert \hspace{.1cm} \pi \in \PPF(n)\right) \sim \Phi(x),
\end{equation*}
where $\Phi(x)$ is the standard normal distribution function.
\end{proposition}

\begin{proof}
We proceed as in the proof of Proposition \ref{Poisson-CLT}. From the
proof of Theorem \ref{thm:PPF}, the probability generating function of
$\boldsymbol{X}:=(U(\pi), R(\pi))$ may be decomposed into
$\boldsymbol{X}=\sum_{i=1}^{n-1}\boldsymbol{X}_i$, with
$\boldsymbol{X}_i=(U_i(\pi), R_i(\pi))$ independent, $U_i$ and $R_i$ both
Bernoulli, and 
\begin{equation*}
\ER(U_i)=\frac{i}{n-1}, \hspace{1cm} \ER(R_i)=\frac{1}{n-1},
\end{equation*}
\begin{equation*}
\Var(U_i)=\frac{i}{n-1}\left(1-\frac{i}{n-1}\right), \hspace{1cm}
\Var(R_i)=\frac{1}{n-1}\left(1-\frac{1}{n-1}\right). 
\end{equation*}
We compute
\begin{equation*}
\sum_{i=1}^{n-1} \ER(U_i) \sim \frac{n}{2}, \hspace{1cm}
\sum_{i=1}^{n-1} \ER(R_i) \sim 1, \hspace{1cm} 
\sum_{i=1}^{n-1} \Var(U_i) \sim \frac{n}{6}, \hspace{1cm}
\sum_{i=1}^{m-1} \Var(R_i) \sim 1. 
\end{equation*}
We recognize that $R(\pi)$ may be approximated by Poisson$(1)$ while
$U(\pi)$ may be approximated by normal$(0, 1)$ after
standardization. From Theorem \ref{thm:PPF-2} and Remark
\ref{rk:PPF-2}, the distribution of $L_s(\pi)$ similarly follows. 
\end{proof}

Even though results concerning parking statistics for classical
parking functions and prime parking functions seem largely parallel
till now, the next theorem shows that unlike classical parking
functions, the distributions of the leading elements statistic and the
$1$'s statistic are not the same for prime parking functions. 

\begin{theorem}
\begin{multline}\label{ppf_leading_1}
\sum_{\pi \in \PPF(n+1)} x^{\lel(\pi)} y^{\one(\pi)}\\
=x(n-1+x+y)^{n-1}\left((n-1)y-x-(n-1)\right)+x(n-1+x)^n+x^2y^2(n+xy)^{n-1}.
\end{multline}
\end{theorem}

\begin{proof}
In Kalikow \cite[Section 1.7.2]{Kalikow}, a bijection between $(p, i)$
where $p \in \PPF(n+1)$ and $p(i)=1$ and $(q, i)$ where $q \in \PF(n)$
and $i\in \{1, \dots, n+1\}$ was identified. The bijection is easy to
describe: we insert an extra $1$ anywhere into a parking function $q$
to obtain a prime parking function $p$, and then divide by the number
of $1$'s in the prime parking function $p$ so that we do not
overcount. 

Borrowing notation from Proposition \ref{thm:master}, this implies that
\begin{align*}
&\sum_{\pi \in \PPF(n+1)} x^{\lel(\pi)} y^{\one(\pi)} \notag \\
=&\sum_{s=0}^{n-2} \sum_{t=0}^{n-s-2} \frac{n}{s+2} \#\{\pi \in
\PF(n)\colon   \pi_1 \neq 1, \hspace{.1cm} \lel(\pi)=t+1 \text{ and }
\one(\pi)=s+1\} x^{t+1} y^{s+2} \notag \\ 
+&\sum_{s=0}^{n-2} \sum_{t=0}^{n-s-2} \frac{1}{s+2} \#\{\pi \in
\PF(n)\colon   \pi_1 \neq 1, \hspace{.1cm} \lel(\pi)=t+1 \text{ and }
\one(\pi)=s+1\} x^{s+2} y^{s+2} \notag \\ 
+&\sum_{s=0}^{n-1} \frac{n+1}{s+2}  \#\{\pi \in \PF(n)\colon   \pi_1=1, \hspace{.1cm}  \lel(\pi)=s+1 \text{ and } \one(\pi)=s+1\} x^{s+2} y^{s+2}.
\end{align*}
Here for $\pi \in \PF(n)$ with $\pi_1 \neq 1$, we classify into two
cases as inserting the extra $1$ in front of all elements in $\pi$
will change the leading element while inserting the extra $1$ anywhere
else will keep the leading element. 

We perform generating function calculations for the summands on the
right. From the proof of Proposition \ref{thm:master}, 
\begin{align*}
&\sum_{s=0}^{n-2} \sum_{t=0}^{n-s-2} \frac{n}{s+2} \binom{n-2}{s, t, n-s-t-2} (n-1)^{n-s-t-1} x^{t+1} y^{s+2} \notag \\
=&\sum_{s=0}^{n-2} \sum_{t=0}^{n-s-2} \frac{1}{n+1}\left((s+2)(t+1)-(t+1)\right) \binom{n+1}{s+2, t+1, n-s-t-2} (n-1)^{n-s-t-2} x^{t+1} y^{s+2} \notag \\
=&\frac{1}{n+1} \left(xy\frac{\partial^2}{\partial x\partial y}(n-1+x+y)^{n+1}-x\frac{\partial}{\partial x}(n-1+x+y)^{n+1}+x\frac{d}{dx}(n-1+x)^{n+1}\right) \notag \\
=&x(n-1+x+y)^{n-1}\left((n-1)y-x-(n-1)\right)+x(n-1+x)^n.
\end{align*}

\begin{align*}
&\sum_{s=0}^{n-2} \sum_{t=0}^{n-s-2} \frac{1}{s+2} \binom{n-2}{s, t, n-s-t-2} (n-1)^{n-s-t-1} (xy)^{s+2} \notag \\
=&\sum_{s=0}^{n-2} \frac{n-1}{s+2} \binom{n-2}{s} (xy)^{s+2} \sum_{t=0}^{n-s-2} \binom{n-s-2}{n-s-t-2} (n-1)^{n-s-t-2} \notag \\
=&\sum_{s=0}^{n-2} \frac{n-1}{s+2} \binom{n-2}{s} (xy)^{s+2} n^{n-s-2} \notag \\
=&\sum_{s=0}^{n-2} \frac{(s+2)-1}{n} \binom{n}{s+2} (xy)^{s+2} n^{n-s-2} \notag \\
=&\frac{1}{n} \left(xy\frac{d}{d(xy)} (n+xy)^n-(n+xy)^n\right)+n^{n-1} \notag \\
=&\frac{1}{n}(n+xy)^{n-1}((n-1)xy-n)+n^{n-1}.
\end{align*}

\begin{align*}
&\sum_{s=0}^{n-1} \frac{n+1}{s+2} \binom{n-1}{s} (s+1) n^{n-s-2} (xy)^{s+2} \notag \\
=&\sum_{s=0}^{n-1} \frac{(s+2)(s+1)-(s+2)+1}{n^2} \binom{n+1}{s+2} n^{n-s-1} (xy)^{s+2} \notag \\
=&\frac{1}{n^2} \left(x^2y^2 \frac{d^2}{d^2(xy)} (n+xy)^{n+1}-xy \frac{d}{d(xy)} (n+xy)^{n+1}+(n+xy)^{n+1}\right)-n^{n-1} \notag \\
=&\frac{1}{n} (n+xy)^{n-1} (nx^2y^2-(n-1)xy+n)-n^{n-1}.
\end{align*}
Putting all this together, the desired result is obtained.
\end{proof}

\begin{corollary}
Taking $x=1$ in (\ref{ppf_leading_1}), we have
\begin{equation*}
\sum_{\pi \in \PPF(n+1)} y^{\one(\pi)}=(n+y)^{n}(y-1)+n^n.
\end{equation*}
This agrees with the corresponding formula in Kalikow \cite[Section
  1.7.2]{Kalikow}. 
\end{corollary}

\subsection{Unit interval parking functions} Unit interval parking
functions are a subset of classical parking functions where each car
is displaced by at most $1$ spot \cite{CHJRSS}. Here we consider the deterministic
model where a car moves forward when its desired spot is
occupied. Denote the set of unit interval parking functions of length
$n$ by $\UPF(n)$. If a car is displaced by $1$ after parking then it
is unlucky; otherwise it is lucky. Let $P_0(y)=1$
and
  $$P_n(y)=\sum_{\pi \in \UPF(n)} y^{\#\{i\colon  \text{ displacement of
    car $i$ is $1$}\}}.$$ 

We recognize that $P_n(y)=\sum_{k=1}^n S(n, k)k! y^{n-k},$
where $S(n, k)$ is a Stirling number of the second kind. One way to
see this is by partitioning the parking preference sequence $\pi \in
\UPF(n)$ into $k \leq n$ noninteracting consecutive blocks while
reading $\pi$ from left to right. The parking outcomes of different
blocks have no effect on one another, and the first number assigned to
each block corresponds to a lucky car. For example, for $\pi=(4, 1,
1, 2, 6, 4)$, there are three blocks: $\{4, 4\}$, $\{1, 1, 2\}$, and
$\{6\}$. We see that car $1$ (with preference $4$), car $2$ (with
preference $1$), and car $5$ (with preference $6$) are lucky. Note
that $P_n(1)$ gives the Fubini numbers (ordered Bell numbers, or
number of ordered set partitions). 

\begin{theorem}\label{thm:UPF}
$P_n(y)$ satisfies the following recurrence relation:
\begin{equation*}
P_n(y)=(y+1)\sum_{i=0}^{n-1} \binom{n-1}{i} P_i(y) P_{n-i-1}(y)-yP_{n-1}(y).
\end{equation*}
\end{theorem}

\begin{proof}
As in the proof of Theorem \ref{un1-general}, we focus on the last
car. The parking protocol implies that $(\pi_1, \dots, \pi_{n-1})$ may
be decomposed into two unit interval parking functions $\alpha \in
\UPF(i)$ and $\beta \in \UPF(n-1-i)$, where $0\leq i\leq n-1$,
$\betaR$ is formed by subtracting $i+1$ from the relevant entries in
$\pi$, and $\alpha$ and $\beta$ do not interact with each other. For
$1\leq i\leq n-1$, this open spot $i+1$ could be either the same as
$j$, the preference of the last car, in which case the car parks
directly. Or, $i+1$ could be equal to $j+1$, in which case the car
travels forward to park. Combined, parking the last car contributes
$y+1$ to the generating function. We note that when $i=0$ (and so
$i+1=1$), the preference of the last car has to coincide with the open
spot. Subtracting $yP_{n-1}(y)$ from the generating function takes
care of this special case. 
\end{proof}

Let $Q(y, t)=\sum_{n=0}^\infty P_n(y) \frac{t^n}{n!}$ and $Q(y, 0)=1$. From Theorem \ref{thm:UPF}, differentiating with respect to $t$ yields
\begin{equation*}
\frac{\partial Q}{\partial t}=(y+1)Q^2-yQ.
\end{equation*}
This is known as a Bernoulli differential equation. To solve it, we follow standard techniques.
\begin{equation*}
Q^{-2} \frac{\partial Q}{\partial t}+Q^{-1} y=y+1.
\end{equation*}
We make a substitution $R=Q^{-1}$. Then
\begin{equation*}
-\frac{\partial R}{\partial t}+Ry=y+1.
\end{equation*}
Hence 
\begin{equation*}
R(y, t)=1+\frac{1}{y}\left(1-\exp(yt)\right) \text{ and } Q(y, t)=\frac{y}{y+1-\exp(yt)}.
\end{equation*}
Setting $y=1$, $Q(1, t)=1/(2-\exp(t))$ gives the exponential
generating function of Fubini numbers.

\subsection{Generic \boldmath{$u$}-parking functions}

Given a positive-integer-valued vector $\boldsymbol{u}=(u_1, \dots, u_m)$ with
$u_1<\cdots<u_m$, a $\boldsymbol{u}$-parking function of length $m$ is a sequence
$\pi=(\pi_1, \dots, \pi_m)$ of positive integers whose (weakly) increasing
rearrangement $(\lambda_1, \dots, \lambda_m)$ satisfies $\lambda_i\leq
u_i$ for all $1\leq i\leq m$. Recurrence relations for generalized
displacement in $\boldsymbol{u}$-parking functions have been widely studied; see
Yan \cite[Section 1.4.3]{Yan} for a summary of results. Let us derive
the recurrence relations for some other parking statistics under the
deterministic model. 

Let
\begin{equation*}
A_{\boldsymbol{u}}(x)= \sum_{\pi \in \PF(\boldsymbol{u})} x^{\unl(\pi)}, \hspace{.2cm}
B_{\boldsymbol{u}}(y)= \sum_{\pi \in \PF(\boldsymbol{u})} y^{\one(\pi)}, \hspace{.2cm}
\text{ and }C_{\boldsymbol{u}}(z)= \sum_{\pi \in \PF(\boldsymbol{u})} z^{\lel(\pi)}.
\end{equation*}

\begin{theorem}\label{un1}
$A_{\boldsymbol{u}}(x)$, $B_{\boldsymbol{u}}(y)$, and
  $C_{\boldsymbol{u}}(z)$ respectively satisfy the following
  recurrence relations: 
\begin{equation*}
A_{\boldsymbol{u}}(x)=\sum_{i=0}^{m-1} \binom{m-1}{i} (ix+(u_{i+1}-i))
A_{\boldsymbol{u}_1}(x) A_{\boldsymbol{u}_2}(x), 
\end{equation*}
\begin{equation*}
B_{\boldsymbol{u}}(y)=\sum_{i=0}^{m-1} \binom{m-1}{i} (y+(u_{i+1}-1))
B_{\boldsymbol{u}_1}(y) B_{\boldsymbol{u}_2}(1), 
\end{equation*}
\begin{equation*}
C_{\boldsymbol{u}}(z)=\sum_{i=0}^{m-1} \left( \binom{m-2}{i-1}
(z+(u_{i+1}-1)) C_{\boldsymbol{u}_1}(z)
C_{\boldsymbol{u}_2}(1)+\binom{m-2}{i} u_{i+1} C_{\boldsymbol{u}_1}(1)
C_{\boldsymbol{u}_2}(z) \right), 
\end{equation*}
where $\boldsymbol{u}_1=(u_1,\dots,u_i)$ and
$\boldsymbol{u}_2=(u_{i+2}-u_{i+1},\dots,u_m-u_{i+1})$. 
\end{theorem}


\begin{proof}
As in the proof of Theorem \ref{un1-general}, we focus on the last
car. We assume $\pi_m$ is the maximal parking completion for $\pi_1,
\dots, \pi_{m-1}$, i.e., $\pi=(\pi_1, \dots, \pi_{m-1}, \pi_m)$ is a
$\boldsymbol{u}$-parking function but $\pi'=(\pi_1, \dots, \pi_{m-1}, \pi_m+1)$
is not a $\boldsymbol{u}$-parking function. Let $\lambda=(\lambda_1, \dots,
\lambda_m)$ be the increasing rearrangement of $\pi$. Then for
some $i$ where $0\leq i\leq m-1$, we have $\pi_m=\lambda_{i+1} \leq
u_{i+1}$. If there are multiple entries in $\pi$ (and hence
$\lambda$) that take the same value as $\pi_m$, we may assume that
$i+1$ is the maximum such index, which implies that
$\lambda_{i+1}<\lambda_{i+2}$ if $i+1<m$. We claim that
$\lambda_{i+1}=u_{i+1}$. Suppose otherwise and that
$\lambda_{i+1}<u_{i+1}$; then the increasing rearrangement of
$\pi'$ is $(\lambda_1, \dots, \lambda_{i}, \lambda_{i+1}+1,
\lambda_{i+2}, \dots, \lambda_m)$, making $\pi'$ a $\boldsymbol{u}$-parking
function, which contradicts the assumption that $\pi_m$ is the maximal
parking completion. Since $u_1<\cdots<u_m$, we further have $\pi_m$ is
the unique entry in $\pi$ with $\pi_m=u_{i+1}$. Hence $\pi$ may be
decomposed into two parking functions: $\alphaR \in \PF(\boldsymbol{u}_1)$ and
$\betaR \in \PF(\boldsymbol{u}_2)$ where $\boldsymbol{u}_1=(u_1,\dots,u_i)$ and
$\boldsymbol{u}_2=(u_{i+2}-u_{i+1},\dots,u_m-u_{i+1})$, $\betaR$ is formed by
subtracting $u_{i+1}$ from the relevant entries in $\pi$, and
$\alphaR$ and $\betaR$ do not interact with each other. 

Now in the general case, $\pi_m$ is a parking completion for $\pi_1,
\dots, \pi_{m-1}$ but not necessarily maximal, so $1\leq \pi_m \leq
u_{i+1}$. We have 
\begin{equation*}
\unl(\pi)=\unl(\alphaR)+\unl(\betaR)+(\text{either } 1 \text{ or } 0),
\end{equation*}
where it is $1$ if $\pi_m$ coincides with any of the $i$ spots already
parked in by the $i$ cars that constitute $\alphaR$ and $0$
otherwise. We also have
\begin{equation*}
 \one(\pi)=\#\{i\colon  \alpha_i=1\}+(\text{either } 1 \text{ or } 0),
\end{equation*}
where it is $1$ if $\pi_m=1$ and $0$ otherwise. For both the
unluckiness and the $1$'s statistics, the binomial coefficient
$\binom{m-1}{i}$ chooses an $i$-element subset of $[m-1]$ to
constitute the index set of $\alphaR$. For the leading elements
statistics, depending on whether $\pi_1<u_{i+1}$ or $\pi_1>u_{i+1}$
(note that it is impossible for $\pi_1=u_{i+1}$), the count on
$\lel(\pi)$ would be different, as $\pi_m$ has the possibility of
equalling $\pi_1$ in the former case but no possibility of equalling $\pi_1$ in
the latter case. In the former case, the binomial coefficient
$\binom{m-2}{i-1}$ chooses an $(i-1)$-element subset of
$[m-1]\setminus\{1\}$, and together with the index $1$, they
constitute the index set of $\alphaR$, and 
\begin{equation*}
\lel(\pi)=\#\{i\colon  \alpha_i=\pi_1\}+(\text{either } 1 \text{ or } 0),
\end{equation*}
where it is $1$ if $\pi_m=\pi_1$ and $0$ otherwise. In the latter
case, the binomial coefficient $\binom{m-2}{i}$ chooses an $i$-element
subset of $[m-1]\setminus\{1\}$ to constitute the index set of
$\alphaR$, and 
\begin{equation*}
\lel(\pi)=\#\{i\colon  \beta_i=\pi_1-u_{i+1}\}.
\end{equation*}
\end{proof}

\appendix

\begin{table}%
\makebox[\linewidth]{
    \begin{tabular}{|c|c|c|c|c|}
    \hline
    \begin{minipage}{.2\textwidth}
    \begin{center}
    Pr\"{u}fer code
    \end{center}
    \end{minipage}
    &
    \begin{minipage}{.2\textwidth}
    \begin{center}
    labelled tree
\end{center}
\end{minipage}
&
    \begin{minipage}{.2\textwidth}
    \begin{center}
    parking function (repeats)
    \end{center}
    \end{minipage}
    &
    \begin{minipage}{.2\textwidth}
    \begin{center}
    parking function (leading elements)
    \end{center}
    \end{minipage}
    &
    \begin{minipage}{.2\textwidth}
    \begin{center}
    parking function (1's)
    \end{center}
    \end{minipage}\\
    \hline
    \begin{minipage}{.2\textwidth}
    \begin{center}
    00
    \end{center}
    \end{minipage}
    &
    \begin{minipage}{.2\textwidth}
    \begin{center}
    \begin{tikzpicture}[scale=0.6]
\node at (0,0) {0} [grow = down]
    child {node{1} edge from parent [thick]}
    child {node{2} edge from parent [thick]}
    child {node{3} edge from parent [thick]};
\end{tikzpicture}
\end{center}
\end{minipage}
&
    \begin{minipage}{.2\textwidth}
    \begin{center}
    111
    \end{center}
    \end{minipage}
    &
    \begin{minipage}{.2\textwidth}
    \begin{center}
    111
    \end{center}
    \end{minipage}
    &
    \begin{minipage}{.2\textwidth}
    \begin{center}
    111
    \end{center}
    \end{minipage}\\
    \hline
    \begin{minipage}{.2\textwidth}
    \begin{center}
    01
    \end{center}
    \end{minipage}
    &
    \begin{minipage}{.2\textwidth}
    \begin{center}
    \begin{tikzpicture}[scale=0.6]
\node at (0,0) {0} [grow = down]
    child {node{2} edge from parent [thick]}
    child {node{1} edge from parent [thick] child {node{3} edge from parent [thick]}};
\end{tikzpicture}
\end{center}
\end{minipage}
    &
    \begin{minipage}{.2\textwidth}
    \begin{center}
    112
    \end{center}
    \end{minipage}
    &
    \begin{minipage}{.2\textwidth}
    \begin{center}
    112
    \end{center}
    \end{minipage}
    &
    \begin{minipage}{.2\textwidth}
    \begin{center}
    112
    \end{center}
    \end{minipage}\\
    \hline
    \begin{minipage}{.2\textwidth}
    \begin{center}
    02
    \end{center}
    \end{minipage}
    &
    \begin{minipage}{.2\textwidth}
    \begin{center}
    \begin{tikzpicture}[scale=0.6]
\node at (0,0) {0} [grow = down]
    child {node{1} edge from parent [thick]}
    child {node{2} edge from parent [thick] child {node{3} edge from parent [thick]}};
\end{tikzpicture}
\end{center}
\end{minipage}
&
    \begin{minipage}{.2\textwidth}
    \begin{center}
    113
    \end{center}
    \end{minipage}
    &
    \begin{minipage}{.2\textwidth}
    \begin{center}
    113
    \end{center}
    \end{minipage}
    &
    \begin{minipage}{.2\textwidth}
    \begin{center}
    113
    \end{center}
    \end{minipage}\\
    \hline
    \begin{minipage}{.2\textwidth}
    \begin{center}
    03
    \end{center}
    \end{minipage}
    &
    \begin{minipage}{.2\textwidth}
    \begin{center}
    \begin{tikzpicture}[scale=0.6]
\node at (0,0) {0} [grow = down]
    child {node{1} edge from parent [thick]}
    child {node{3} edge from parent [thick] child {node{2} edge from parent [thick]}};
\end{tikzpicture}
\end{center}
\end{minipage}
&
    \begin{minipage}{.2\textwidth}
    \begin{center}
    221
    \end{center}
    \end{minipage}
    &
    \begin{minipage}{.2\textwidth}
    \begin{center}
    221
    \end{center}
    \end{minipage}
    &
    \begin{minipage}{.2\textwidth}
    \begin{center}
    131
    \end{center}
    \end{minipage}\\
    \hline
    \begin{minipage}{.2\textwidth}
    \begin{center}
    10
    \end{center}
    \end{minipage}
    &
    \begin{minipage}{.2\textwidth}
    \begin{center}
    \begin{tikzpicture}[scale=0.6]
\node at (0,0) {0} [grow = down]
    child {node{3} edge from parent [thick]}
    child {node{1} edge from parent [thick] child {node{2} edge from parent [thick]}};
\end{tikzpicture}
\end{center}
\end{minipage}
&
    \begin{minipage}{.2\textwidth}
    \begin{center}
    122
    \end{center}
    \end{minipage}
    &
    \begin{minipage}{.2\textwidth}
    \begin{center}
    121
    \end{center}
    \end{minipage}
    &
    \begin{minipage}{.2\textwidth}
    \begin{center}
    121
    \end{center}
    \end{minipage}\\
    \hline
    \begin{minipage}{.2\textwidth}
    \begin{center}
    11
    \end{center}
    \end{minipage}
    &
    \begin{minipage}{.2\textwidth}
    \begin{center}
    \begin{tikzpicture}[scale=0.6]
\node at (0,0) {0} [grow = down]
    child {node{1} edge from parent [thick] child {node{2} edge from parent [thick]} child {node{3} edge from parent [thick]}};
\end{tikzpicture}
\end{center}
\end{minipage}
&
    \begin{minipage}{.2\textwidth}
    \begin{center}
    123
    \end{center}
    \end{minipage}
    &
    \begin{minipage}{.2\textwidth}
    \begin{center}
    122
    \end{center}
    \end{minipage}
    &
    \begin{minipage}{.2\textwidth}
    \begin{center}
    122
    \end{center}
    \end{minipage}\\
    \hline
    \begin{minipage}{.2\textwidth}
    \begin{center}
    12
    \end{center}
    \end{minipage}
    &
    \begin{minipage}{.2\textwidth}
    \begin{center}
    \begin{tikzpicture}[scale=0.6]
\node at (0,0) {0} [grow = down]
    child {node{1} edge from parent [thick] child {node{2} edge from parent [thick] child {node{3} edge from parent [thick]}}};
\end{tikzpicture}
\end{center}
\end{minipage}
&
    \begin{minipage}{.2\textwidth}
    \begin{center}
    231
    \end{center}
    \end{minipage}
    &
    \begin{minipage}{.2\textwidth}
    \begin{center}
    123
    \end{center}
    \end{minipage}
    &
    \begin{minipage}{.2\textwidth}
    \begin{center}
    123
    \end{center}
    \end{minipage}\\
    \hline
    \begin{minipage}{.2\textwidth}
    \begin{center}
    13
    \end{center}
    \end{minipage}
    &
    \begin{minipage}{.2\textwidth}
    \begin{center}
    \begin{tikzpicture}[scale=0.6]
\node at (0,0) {0} [grow = down]
    child {node{1} edge from parent [thick] child {node{3} edge from parent [thick] child {node{2} edge from parent [thick]}}};
\end{tikzpicture}
\end{center}
\end{minipage}
&
    \begin{minipage}{.2\textwidth}
    \begin{center}
    121
    \end{center}
    \end{minipage}
    &
    \begin{minipage}{.2\textwidth}
    \begin{center}
    231
    \end{center}
    \end{minipage}
    &
    \begin{minipage}{.2\textwidth}
    \begin{center}
    132
    \end{center}
    \end{minipage}\\
    \hline
\end{tabular}
}
\caption{Correspondence table for $n=3$.}
\label{tree-parking:illustration}
\end{table}

\begin{table}%
\makebox[\linewidth]{
    \begin{tabular}{|c|c|c|c|c|}
    \hline
    \begin{minipage}{.2\textwidth}
    \begin{center}
    20
    \end{center}
    \end{minipage}
    &
    \begin{minipage}{.2\textwidth}
    \begin{center}
    \begin{tikzpicture}[scale=0.6]
\node at (0,0) {0} [grow = down]
    child {node{3} edge from parent [thick]}
    child {node{2} edge from parent [thick] child {node{1} edge from parent [thick]}};
\end{tikzpicture}
\end{center}
\end{minipage}
&
    \begin{minipage}{.2\textwidth}
    \begin{center}
    311
    \end{center}
    \end{minipage}
    &
    \begin{minipage}{.2\textwidth}
    \begin{center}
    131
    \end{center}
    \end{minipage}
    &
    \begin{minipage}{.2\textwidth}
    \begin{center}
    211
    \end{center}
    \end{minipage}\\
    \hline
    \begin{minipage}{.2\textwidth}
    \begin{center}
    21
    \end{center}
    \end{minipage}
      &
    \begin{minipage}{.2\textwidth}
    \begin{center}
    \begin{tikzpicture}[scale=0.6]
\node at (0,0) {0} [grow = down]
    child {node{2} edge from parent [thick] child {node{1} edge from parent [thick] child {node{3} edge from parent [thick]}}};
\end{tikzpicture}
\end{center}
\end{minipage}
&
    \begin{minipage}{.2\textwidth}
    \begin{center}
    312
    \end{center}
    \end{minipage}
    &
    \begin{minipage}{.2\textwidth}
    \begin{center}
    132
    \end{center}
    \end{minipage}
    &
    \begin{minipage}{.2\textwidth}
    \begin{center}
    213
    \end{center}
    \end{minipage}\\
    \hline
    \begin{minipage}{.2\textwidth}
    \begin{center}
    22
    \end{center}
    \end{minipage}
     &
    \begin{minipage}{.2\textwidth}
    \begin{center}
    \begin{tikzpicture}[scale=0.6]
\node at (0,0) {0} [grow = down]
    child {node{2} edge from parent [thick] child{node{1} edge from parent [thick]} child {node{3} edge from parent [thick]}};
\end{tikzpicture}
\end{center}
\end{minipage}
&
    \begin{minipage}{.2\textwidth}
    \begin{center}
    131
    \end{center}
    \end{minipage}
    &
    \begin{minipage}{.2\textwidth}
    \begin{center}
    311
    \end{center}
    \end{minipage}
    &
    \begin{minipage}{.2\textwidth}
    \begin{center}
    212
    \end{center}
    \end{minipage}\\
    \hline
    \begin{minipage}{.2\textwidth}
    \begin{center}
    23
    \end{center}
    \end{minipage}
      &
    \begin{minipage}{.2\textwidth}
    \begin{center}
    \begin{tikzpicture}[scale=0.6]
\node at (0,0) {0} [grow = down]
    child {node{2} edge from parent [thick] child {node{3} edge from parent [thick] child {node{1} edge from parent [thick]}}};
\end{tikzpicture}
\end{center}
\end{minipage}
&
    \begin{minipage}{.2\textwidth}
    \begin{center}
    132
    \end{center}
    \end{minipage}
    &
    \begin{minipage}{.2\textwidth}
    \begin{center}
    312
    \end{center}
    \end{minipage}
    &
    \begin{minipage}{.2\textwidth}
    \begin{center}
    312
    \end{center}
    \end{minipage}\\
    \hline
    \begin{minipage}{.2\textwidth}
    \begin{center}
    30
    \end{center}
    \end{minipage}
     &
    \begin{minipage}{.2\textwidth}
    \begin{center}
    \begin{tikzpicture}[scale=0.6]
\node at (0,0) {0} [grow = down]
    child {node{2} edge from parent [thick]}
    child {node{3} edge from parent [thick] child {node{1} edge from parent [thick]}};
\end{tikzpicture}
\end{center}
\end{minipage}
&
    \begin{minipage}{.2\textwidth}
    \begin{center}
    211
    \end{center}
    \end{minipage}
    &
    \begin{minipage}{.2\textwidth}
    \begin{center}
    212
    \end{center}
    \end{minipage}
    &
    \begin{minipage}{.2\textwidth}
    \begin{center}
    311
    \end{center}
    \end{minipage}\\
    \hline
    \begin{minipage}{.2\textwidth}
    \begin{center}
    31
    \end{center}
    \end{minipage}
     &
    \begin{minipage}{.2\textwidth}
    \begin{center}
    \begin{tikzpicture}[scale=0.6]
\node at (0,0) {0} [grow = down]
    child {node{3} edge from parent [thick] child {node{1} edge from parent [thick] child {node{2} edge from parent [thick]}}};
\end{tikzpicture}
\end{center}
\end{minipage}
&
    \begin{minipage}{.2\textwidth}
    \begin{center}
    212
    \end{center}
    \end{minipage}
    &
    \begin{minipage}{.2\textwidth}
    \begin{center}
    213
    \end{center}
    \end{minipage}
    &
    \begin{minipage}{.2\textwidth}
    \begin{center}
    231
    \end{center}
    \end{minipage}\\
    \hline
    \begin{minipage}{.2\textwidth}
    \begin{center}
    32
    \end{center}
    \end{minipage}
     &
    \begin{minipage}{.2\textwidth}
    \begin{center}
    \begin{tikzpicture}[scale=0.6]
\node at (0,0) {0} [grow = down]
    child {node{3} edge from parent [thick] child {node{2} edge from parent [thick] child {node{1} edge from parent [thick]}}};
\end{tikzpicture}
\end{center}
\end{minipage}
&
    \begin{minipage}{.2\textwidth}
    \begin{center}
    213
    \end{center}
    \end{minipage}
    &
    \begin{minipage}{.2\textwidth}
    \begin{center}
    321
    \end{center}
    \end{minipage}
    &
    \begin{minipage}{.2\textwidth}
    \begin{center}
    321
    \end{center}
    \end{minipage}\\
    \hline
    \begin{minipage}{.2\textwidth}
    \begin{center}
    33
    \end{center}
    \end{minipage}
     &
    \begin{minipage}{.2\textwidth}
    \begin{center}
    \begin{tikzpicture}[scale=0.6]
\node at (0,0) {0} [grow = down]
    child {node{3} edge from parent [thick] child {node{1} edge from parent [thick]} child {node{2} edge from parent [thick]}};
\end{tikzpicture}
\end{center}
\end{minipage}
&
    \begin{minipage}{.2\textwidth}
    \begin{center}
    321
    \end{center}
    \end{minipage}
    &
    \begin{minipage}{.2\textwidth}
    \begin{center}
    211
    \end{center}
    \end{minipage}
    &
    \begin{minipage}{.2\textwidth}
    \begin{center}
    221
    \end{center}
    \end{minipage}\\
    \hline
    \end{tabular}
    }
\end{table}


\begin{thebibliography}{30}
\bibitem{AR} Armstrong, D., Rhoades, B.: The Shi arrangement and the
  Ish arrangement. \textit{Trans.\ Amer.\ Math.\ Soc.}\ 364: 1509-1528
  (2011).  

\bibitem{CHJRSS} Chaves Meyles, L., Harris, P.E., Jordaan, R., Rojas Kirby, G., 
Sehayek, S., Spingarn, E.: Unit-interval parking functions and the permutohedron.
arXiv: 2305.15554 (2023).

\bibitem{DH} Diaconis, P., Hicks, A.: Probabilizing parking
  functions. \textit{Adv.\ Appl.\ Math.}\ 89: 125-155 (2017). 

\bibitem{DHHRY} Durmi\'{c}, I., Han, A., Harris, P.E., Ribeiro, R.,
  Yin, M.: Probabilistic parking functions. arXiv: 2211.00536 (2022). 

\bibitem{Pollak} Foata, D., Riordan, J.: Mappings of acyclic and
  parking functions. \textit{Aequationes Math.}\ 10: 10-22 (1974). 

\bibitem{GS} Gessel, I.M., Seo, S.: A refinement of Cayley's formula
  for trees. \textit{Electron.\ J. Combin.}\ 11: Research Paper 27, 23
  pp. (2006). 

\bibitem{Hou} Hou, Q.-H.: An insertion algorithm and leaders of rooted
  trees. \textit{European J. Combin.}\ 53: 35-44 (2016). 

\bibitem{Kalikow} Kalikow, L.H.: Enumeration of parking functions,
  allowable permutation pairs, and labeled trees. Ph.D. Thesis,
  Brandeis University (1999). 

\bibitem{KY} Kenyon, R., Yin, M.: Parking functions: From
  combinatorics to
  probability. \textit{Methodol.\ Comput.\ Appl.\ Probab.}\ 25: 32
  (2023). 

\bibitem{k-w} Konheim, A.G., Weiss, B.: An occupancy discipline and
  applications. \textit{SIAM J. Appl.\ Math.}\ 14: 1266-1274 (1966). 

\bibitem{pyke} Pyke, R.: The supremum and infimum of the Poisson
  process. \textit{Ann.\ Math.\ Statist.}\\ 30: 568--576 (1959).
  
\bibitem{Riordan} Riordan, J.: Combinatorial Identities. John Wiley \&
  Sons, Inc., New York. (1968). 

\bibitem{Pollak2} Riordan, J.: Ballots and
  trees. \textit{J. Combin.\ Theory} 6: 408-411 (1969).

\bibitem{Seo} Seo, S.: A combinatorial proof of Postnikov's identity
  and a generalized enumeration of labeled
  trees. \textit{Electron.\ J. Combin.}\ 11: Note 3, 9 pp. (2004). 

\bibitem{SS} Seo, S., Shin, H.: A generalized enumeration of labeled
  trees and reverse Pr\"{u}fer algorithm. \text{J. Combin.\ Theory
    Ser.~A} 114: 1357-1361 (2007). 

\bibitem{Stanley1} Stanley, R.P.: Parking functions and noncrossing
  partitions. \textit{Electron.\ J. Combin.}\ 4: Research Paper 20, 14
  pp. (1997). 

\bibitem{Stanley} Stanley, R.P.: Enumerative Combinatorics Volume
  2. Cambridge University Press, Cambridge. (1999). 

\bibitem{Stanley2} Stanley, R.P.: An introduction to hyperplane
  arrangements. In: Miller, E., Reiner, V., Sturmfels, B. (eds.)
  Geometric Combinatorics. IAS/Park City Mathematics Series, Volume
  13, pp.~389-496. American Mathematical Society, Providence. (2007). 

\bibitem{SW} Stanley, R.P., Wang, Y.: Some aspects of $(r,k)$-parking
  functions. \textit{J. Combin.\ Theory Ser.~A} 159: 54-78 (2018). 

\bibitem{Yan} Yan, C.H.: Parking functions. In: B\'{o}na, M. (ed.)
  Handbook of Enumerative Combinatorics. Discrete Math.\ Appl.,
  pp. 835-893. CRC Press, Boca Raton. (2015). 
\end{thebibliography}
\end{document}